\documentclass[11pt]{amsart}
\usepackage{amssymb,amsfonts,amsthm,amscd,amsmath,amsgen,upref}

\usepackage[margin=1in]{geometry}

\theoremstyle{plain}
\newtheorem{cor}{Corollary}

\newtheorem{prop}[cor]{Proposition}
\newtheorem{con}[cor]{Control}
\newtheorem{thm}[cor]{Theorem}
\theoremstyle{definition}

\numberwithin{cor}{section}
\numberwithin{equation}{section}

\DeclareMathOperator{\tr}{tr}

\DeclareMathOperator{\C}{C}
\DeclareMathOperator{\USC}{USC}
\DeclareMathOperator{\LSC}{LSC}
\DeclareMathOperator{\BUC}{BUC}
\DeclareMathOperator{\Lip}{Lip}

\DeclareMathOperator{\Supp}{Supp}

\renewcommand{\d}{d} 

\newcommand{\abs}[1]{\lvert#1\rvert}
\newcommand{\norm}[1]{\lVert#1\rVert}

\def\XXint#1#2#3{{\setbox0=\hbox{$#1{#2#3}{\int}$ }
\vcenter{\hbox{$#2#3$ }}\kern-.6\wd0}}

\title{On The Existence of an Invariant Measure for Isotropic Diffusions in Random Environment}

\author{Benjamin J. Fehrman}

\date{April 21, 2014}



\address{Department of Mathematics, The University of Chicago, 5734 S. University Avenue, Chicago IL, 60637.}

\email{bfehrman@math.uchicago.edu}

\begin{document}

\begin{abstract}
The results of this paper build upon those first obtained by Sznitman and Zeitouni in \cite{SZ}.  We establish, for spacial dimensions $d\geq 3$, the existence of a unique invariant measure for isotropic diffusions in random environment on $\mathbb{R}^d$ which are small perturbations of Brownian motion.  Furthermore, we establish a general homogenization result for initial data which are locally measurable with respect to the coefficients.
\end{abstract}

\maketitle

\section{Introduction}

The results of this paper should be seen as an extension of those first obtained in Sznitman and Zeitouni \cite{SZ} for stationary diffusion processes in random environment on $\mathbb{R}^d$, for $d\geq 3$, satisfying a restricted isotropy condition and finite range dependence.  Our framework depends upon an underlying probability space $\left(\Omega,\mathcal{F},\mathbb{P}\right)$, which can be viewed as indexing the collection of all equations or environments described, for each $x\in\mathbb{R}^d$ and $\omega\in\Omega$, by the coefficients \begin{equation}\label{intro_coefficients} A(x,\omega)\in\mathcal{S}(d)\;\;\textrm{and}\;\;b(x,\omega)\in\mathbb{R}^d,\end{equation} for $\mathcal{S}(d)$ the space of $d\times d$ symmetric matrices.

We will in particular assume that the coefficients are stationary, satisfying a finite-range dependence and restricted isotropy condition.  Precisely, there exists a group $\left\{\tau_x:\Omega\rightarrow\Omega\right\}_{x\in\mathbb{R}^d}$ of measure-preserving transformations such that, for each $x,y\in\mathbb{R}^d$ and $\omega\in\Omega$, \begin{equation}\label{i_stationary}  A(x+y,\omega)=A(x,\tau_y\omega)\;\;\textrm{and}\;\;b(x+y,\omega)=b(x,\tau_y\omega).\end{equation}  There exists $R>0$ such that, whenever subsets $A,B\subset\mathbb{R}^d$ satisfy $\d(A,B)\geq R$, the sigma-algebras \begin{equation}\label{intro_finite} \sigma\left(A(x,\omega), b(x,\omega)\;|\;x\in A\right)\;\;\textrm{and}\;\;\sigma\left(A(x,\omega), b(x,\omega)\;|\;x\in B\right)\;\;\textrm{are independent.}\end{equation}  And, for every orthogonal transformation $r:\mathbb{R}^d\rightarrow\mathbb{R}^d$ preserving the coordinate axis, for every $x\in\mathbb{R}^d$, the random variables \begin{equation}\label{intro_isotropy} \left(rb(x,\omega), rA(x,\omega)r^t\right)\;\;\textrm{and}\;\;\left(b(rx,\omega), A(rx,\omega)\right)\;\;\textrm{have the same law.}\end{equation}  We remark that these assumptions are identical to model considered in \cite{SZ} and, are the continuous counterpart of the model first studied in the discrete setting by Bricmont and Kupiainen \cite{BK}.

We observe that the martingale problem, for each $x\in\mathbb{R}^d$ and $\omega\in\Omega$, corresponding to the generator $$\frac{1}{2}\sum_{i,j=1}^da_{ij}(y,\omega)\frac{\partial^2}{\partial_{y_i}\partial_{y_j}}-\sum_{i=1}^db_i(y,\omega)\frac{\partial}{\partial_{y_i}},$$ is well-posed, see Stroock and Varadhan \cite{SV}.  We denote by $P_{x,\omega}$ the corresponding probability measure on the space of continuous paths $\C([0,\infty);\mathbb{R}^d)$ and recall that, almost surely with respect to $P_{x,\omega}$, paths $X_t\in\C([0,\infty);\mathbb{R}^d)$ satisfy the stochastic differential equation $$\left\{\begin{array}{l} dX_t=-b(X_t,\omega)dt+\sigma(X_t,\omega)dB_t, \\ X_0=x,\end{array}\right.$$ for $A(y,\omega)=\sigma(y,\omega)\sigma(y,\omega)^t$, and for $B_t$ a standard Brownian motion under $P_{x,\omega}$ with respect to the canonical right continuous filtration on $\C([0,\infty);\mathbb{R}^d)$. 

We may now describe our main result.  For every measurable subset $E\in\mathcal{F}$, recalling the transformation group appearing in (\ref{i_stationary}), we write \begin{equation}\label{i_initial}P_t(\omega,E)=P_{0,\omega}\left(\tau_{X_t}\omega\in E\right).\end{equation}  And, we recall that the transformation group is said to be ergodic if, whenever $E\in\mathcal{F}$ satisfies, for every $x\in\mathbb{R}^d$, $$\tau_x(E)=E\;\;\textrm{in the measure algebra of}\;\;(\Omega,\mathcal{F},\mathbb{P}),$$ either $\mathbb{P}(E)=0$ or $\mathbb{P}(E)=1$.

\begin{thm}\label{i_main}  There exists a unique probability measure $\pi$ on $(\Omega,\mathcal{F})$ which is absolutely continuous with respect to $\mathbb{P}$ and satisfies, for every $t\geq 0$ and $E\in\mathcal{F}$, $$\pi(E)=\int_{\Omega}P_t(\omega,E)\;d\pi.$$  Furthermore, if $\left\{\tau_x\right\}_{x\in\mathbb{R}^d}$ is ergodic, then $\pi$ is mutually absolutely continuous with respect to $\mathbb{P}$ and defines an ergodic probability measure with respect to the canonical Markov process on $\Omega$ defining (\ref{i_initial}).\end{thm}

We achieve this result by analyzing the long term behavior of solutions $u:\mathbb{R}^d\times[0,\infty)\times\Omega\rightarrow\mathbb{R}$ satisfying \begin{equation}\label{i_eq}\left\{\begin{array}{ll} u_t-\frac{1}{2}\tr(A(y,\omega)D^2u)+b(y,\omega)\cdot Du=0 & \textrm{on}\;\;\mathbb{R}^d\times(0,\infty), \\ u(x,0,\omega)=f(x,\omega) & \textrm{on}\;\;\mathbb{R}^\times\left\{0\right\},\end{array}\right.\end{equation} since, for ${\bf{1}}_E:\Omega\rightarrow\Omega$ the indicator function of $E\in\mathcal{F}$, for $f_E(x,\omega)={\bf{1}}_E(\tau_x\omega),$ if $u_E(x,t,\omega)$ satisfies (\ref{i_eq}) with initial data $f_E(x,\omega)$, then, for each $\omega\in\Omega$ and $t\geq 0$, $$u_E(0,t,\omega)=P_{0,\omega}\left({\bf{1}}_E(\tau_{X_t}\omega)\right)=P_{0,\omega}\left(\tau_{X_t}\omega\in E\right)=P_t(\omega, E).$$

Indeed, along an exponentially increasing sequence of time scales $L_n^2$, see (\ref{L}), the invariant measure $\pi$ is first identified, for every $E\in\mathcal{F}$, as the limit $$\pi(E)=\lim_{n\rightarrow\infty}\mathbb{E}\left(u_E(0,L_n^2,\omega)\right),$$  where we prove that the limit exists in Proposition \ref{o_cauchy} and Proposition \ref{o_identify} and, in Proposition \ref{o_pmeasure}, we prove that $\pi$ defines a probability measure on $(\Omega,\mathcal{F})$ which is absolutely continuous with respect to $\mathbb{P}$.

We then establish an almost sure characterization of $\pi$ along the full limit, as $t\rightarrow\infty$, for a class of subsets $E\in\mathcal{F}$ whose indicator functions satisfy a version of (\ref{intro_finite}), see Proposition \ref{u_localize}.  For such subsets, we prove that, on a subset of full probability depending on $E$, \begin{equation}\label{i_almost}\lim_{t\rightarrow\infty}u_E(0,t,\omega)=\pi(E).\end{equation}  Here, we use crucially the results of \cite{SZ}, where it is shown that, with high probability, there exists a coupling at large length and time scales between the diffusion process generated in environment $\omega$ by coefficients $A(y,\omega)$ and $b(y,\omega)$ and a Brownian motion with deterministic variance, see Control \ref{Holder}.  Notice, however, that this coupling cannot in general provide an effective comparison between solutions of (\ref{i_eq}) and solutions $\overline{u}:\mathbb{R}^d\times[0,\infty)\times\Omega\rightarrow\mathbb{R}$ satisfying the deterministic equation, for $\overline{\alpha}>0$ defined in Theorem \ref{effectivediffusivity}, \begin{equation}\label{i_hom}\left\{\begin{array}{ll} \overline{u}_t-\frac{\overline{\alpha}}{2}\Delta\overline{u}=0& \textrm{on}\;\;\mathbb{R}^d\times(0,\infty), \\ \overline{u}(x,0,\omega)=f(x,\omega) & \textrm{on}\;\;\mathbb{R}^d\times\left\{0\right\},\end{array}\right.\end{equation}  since, for stationary initial data, we expect $$\lim_{t\rightarrow \infty}u(0,t,\omega)=\int_{\Omega}f(0,\omega)\;d\pi\;\;\textrm{and}\;\;\lim_{t\rightarrow\infty}\overline{u}(0,t,\omega)=\int_{\Omega}f(0,\omega)\;d\mathbb{P}=\mathbb{E}(f(0,\omega)).$$  

However, in Proposition \ref{o_main}, this coupling does provide a means by which the solution of (\ref{i_eq}) can be effectively compared, with high probability, on large length and time scales, to a quantity which, for suitable initial data, is nearly constant.  That is, with high probability, we obtain an effective comparison between the solution $u(x,t,\omega)$ of (\ref{i_eq}) at time $L_n^2$ with the solution of (\ref{i_hom}) at time $L_n^2-L_{n-1}^2$ corresponding to initial data $u(x,L_{n-1}^2,\omega)$.

This is essentially to say that $u(x,L_n^2,\omega)$ is an averaged version of $u(x,L_{n-1}^2,\omega)$, where we provide a quantitative version of the averaging in Proposition \ref{u_var} for subsets whose characteristic function satisfies  a version of (\ref{intro_finite}), see Propositions \ref{u_sigma_translate} and \ref{u_localize}.  In combination, the comparison and averaging complete the proof of (\ref{i_almost}).

Finally, in \cite{SZ}, localization estimates for the diffusion in environment $\omega$ are obtained with high probability, see Control \ref{localization}.  We use this localization in Proposition \ref{u_upmain} to upgrade the convergence along the discrete sequence $L_n^2$ to the full limit, as $t\rightarrow\infty$, at the cost of obtaining the convergence on a marginally smaller portion of space.  The proof of invariance and uniqueness then follow by standard arguments, see Proposition \ref{u_invariant} and Theorem \ref{u_thm}.

Furthermore, as an application of Proposition \ref{u_upmain}, we establish a homogenization result for oscillating initial data which are locally measurable with respect to the coefficients.  Precisely, we define, for each $R>0$, the sigma algebra $$\sigma_{B_R}=\sigma\left(A(x,\omega),b(x,\omega)\;|\;x\in B_R\right),$$ and consider functions $f\in L^\infty(\mathbb{R}^d\times\Omega)$ which are stationary with respect to the translation group $\left\{\tau_x\right\}_{x\in\mathbb{R}^d}$ and satisfy $f(0,\omega)\in L^\infty(\Omega,\sigma_{B_R})$, where $L^\infty(\Omega,\sigma_{B_R})$ denotes the space of bounded $\sigma_{B_R}$-measurable functions on $\Omega$.

\begin{thm}\label{i_homogenization}  Suppose that $f\in L^\infty(\mathbb{R}^d\times\Omega)$ and $R>0$ satisfy, for each $x,y\in\mathbb{R}^d$ and $\omega\in\Omega$, $$f(x+y,\omega)=f(x,\tau_y\omega),$$ with $f(0,\omega)\in L^\infty(\Omega,\sigma_{B_R})$.   For each $\epsilon>0$, let $u^\epsilon:\mathbb{R}^d\times[0,\infty)\times\Omega\rightarrow\mathbb{R}$ denote the solution to \begin{equation}\label{i_homogenization_1}\left\{\begin{array}{ll} u^\epsilon_t-\frac{1}{2}\tr(A(x/\epsilon,\omega)D^2u^\epsilon)+\frac{1}{\epsilon}b(x/\epsilon,\omega)\cdot Du^\epsilon=0 & \textrm{on}\;\;\mathbb{R}^d\times(0,\infty), \\ u^\epsilon(x,0,\omega)=f(x/\epsilon,\omega) & \textrm{on}\;\;\mathbb{R}^d\times\left\{0\right\}.\end{array}\right.\end{equation}  There exists a subset of full probability such that, as $\epsilon\rightarrow 0$, $$u^\epsilon\rightarrow\int_{\Omega}f(0,\omega)\;d\pi\;\;\textrm{locally uniformly on}\;\;\mathbb{R}^d\times(0,\infty).$$\end{thm}

These method also apply to equations like (\ref{i_homogenization_1}) involving an oscillating righthand side and, to the analogous time independent problems.  See Theorem \ref{h_1thm} and Theorem \ref{h_2thm}.

We remark that, in the case $b(y,\omega)=0$, the existence of an invariant measure and applications to homogenization were established by Papanicolaou and Varadhan \cite{PV1} and Yurinsky \cite{YVV}.  Furthermore, when equation (\ref{i_eq}) may be rewritten in divergence form, results have been obtained by De Masi, Ferrari, Goldstein and Wick \cite{MFGW}, Kozlov \cite{Kozlov}, Olla \cite{Olla}, Osada \cite{Osada} and Papanicolaou and Varadhan \cite{PV}.  We point the interested reader to the introduction of \cite{SZ} for a more complete list of references regarding related problems in the discrete setting.

The paper is organized as follows.  In Section 2, we present our notation and assumptions as well as provide a summary of the aspects of \cite{SZ} most relevant to our arguments.  We identify the invariant measure in Section 3 and, in Section 4, we prove that the invariant measure is indeed invariant and unique.  Finally, in Section 5, we prove the general homogenization result for functions which are locally measurable with respect to the coefficients.

\subsection*{Acknowledgments}

I would like to thank Professor Panagiotis Souganidis for suggesting this problem,  and I would like to thank Professors Panagiotis Souganidis, Ofer Zeitouni and Luis Silvestre for many useful conversations.

\section{Preliminaries} 

\subsection{Notation}

Elements of $\mathbb{R}^d$ and $[0,\infty)$ are denoted by $x$ and $y$ and $t$ respectively and $(x,y)$ denotes the standard inner product on $\mathbb{R}^d$.  We write $Dv$ and $v_t$ for the derivative of the scalar function $v$ with respect to $x\in\mathbb{R}^d$ and $t\in[0,\infty)$, while $D^2v$ stands for the Hessian of $v$.  The spaces of $k\times l$ and $k\times k$ symmetric matrices with real entries are respectively written $\mathcal{M}^{k\times l}$ and $\mathcal{S}(k)$.  If $M\in\mathcal{M}^{k\times l}$, then $M^t$ is its transpose and $\abs{M}$ is its norm $\abs{M}=\tr(MM^t)^{1/2}.$  If $M$ is a square matrix, we write $\tr(M)$ for the trace of $M$.  The Euclidean distance between subsets $A,B\subset\mathbb{R}^d$ is $$d(A,B)=\inf\left\{\;\abs{a-b}\;|\;a\in A, b\in B\;\right\}$$  and, for an index $\mathcal{A}$ and a family of measurable functions $\left\{f_\alpha:\mathbb{R}^d\times\Omega\rightarrow\mathbb{R}^{n_\alpha}\right\}_{\alpha\in\mathcal{A}}$, we write $$\sigma(f_\alpha(x,\omega)\;|\;x\in A, \alpha\in\mathcal{A})$$ for the sigma algebra generated by the random variables $f_\alpha(x,\omega)$ for $x\in A$ and $\alpha\in\mathcal{A}$.  For $U\subset\mathbb{R}^d$, $\USC(U;\mathbb{R}^d)$, $\LSC(U;\mathbb{R}^d)$, $\BUC(U;\mathbb{R}^d)$, $\Lip(U;\mathbb{R}^d)$, $\C^{0,\beta}(U;\mathbb{R}^d)$ and $\C^k(U;\mathbb{R}^d)$ are the spaces of upper-semicontinuous, lower-semicontinuous, bounded continuous, Lipschitz continuous, $\beta$-H\"{o}lder continuous and $k$-continuously differentiable functions on $U$ with values in $\mathbb{R}^d$.  For $f:\mathbb{R}^d\rightarrow\mathbb{R}$, we write $\Supp(f)$ for the support of $f$.  Furthermore, $B_R$ and $B_R(x)$ are respectively the open balls of radius $R$ centered at zero and $x\in\mathbb{R}^d$.  For a real number $r\in\mathbb{R}$ we write $\left[r\right]$ for the largest integer less than or equal to $r$.  Finally, throughout the paper we write $C$ for constants that may change from line to line but are independent of $\omega\in\Omega$ unless otherwised indicated.

\subsection{The Random Environment}

There exists an underlying probability space $(\Omega,\mathcal{F},\mathbb{P})$ indexing the individual realizations of the random environment.   Since the environment is described, for each $x\in\mathbb{R}^d$ and $\omega\in\Omega$, by the diffusion matrix $A(x,\omega)$ and drift $b(x,\omega)$, we may take \begin{equation}\label{sigmaalgebra} \mathcal{F}=\sigma\left(A(x,\omega),b(x,\omega)\;|\;x\in\mathbb{R}^d\right).\end{equation}  Furthermore, we assume this space is equipped with a \begin{equation}\label{transgroup} \textrm{group of measure-preserving transformations}\; \left\{\tau_x:\Omega\rightarrow\Omega\right\}_{x\in\mathbb{R}^d},\end{equation} such that the coefficients $A:\mathbb{R}^d\times\Omega\rightarrow\mathcal{S}(d)$ and $b:\mathbb{R}^d\times\Omega\rightarrow\mathbb{R}$ are bi-measurable stationary functions satisfying, for each $x,y\in\mathbb{R}^d$ and $\omega\in\Omega$, \begin{equation}\label{stationary} A(x+y,\omega)=A(x,\tau_y\omega)\;\;\textrm{and}\;\;b(x+y,\omega)=b(x,\tau_y\omega).\end{equation}

We assume that the diffusion matrix and drift are bounded and Lipschitz uniformly for $\omega\in\Omega$.  There exists $C>0$ such that, for all $y\in\mathbb{R}^d$ and $\omega\in\Omega$,  \begin{equation}\label{bounded} \abs{b(y,\omega)}\leq C\;\;\;\textrm{and}\;\;\;\abs{A(y,\omega)}\leq C \end{equation} and, for all $x,y\in\mathbb{R}^d$ and $\omega\in\Omega$, \begin{equation}\label{Lipschitz} \abs{b(x,\omega)-b(y,\omega)}\leq C\abs{x-y}\;\;\;\textrm{and}\;\;\;\abs{A(x,\omega)-A(y,\omega)}\leq C\abs{x-y}.\end{equation}  In addition, we assume that the diffusion matrix is uniformly elliptic uniformly in $\Omega$.  There exists $\nu>1$ such that, for all $y\in\mathbb{R}^d$ and $\omega\in\Omega$, \begin{equation}\label{elliptic} \frac{1}{\nu} I\leq A(y,\omega)\leq \nu I.\end{equation}

The coefficients satisfy a finite range dependence.  There exists $R>0$ such that, whenever $A,B\subset\mathbb{R}^d$ satisfy $d(A,B)\geq R$, the sigma algebras \begin{equation}\label{finitedep} \sigma(A(x,\omega), b(x,\omega)\;|\;x\in A)\;\;\;\textrm{and}\;\;\; \sigma(A(x,\omega), b(x,\omega)\;|\;x\in B)\;\;\;\textrm{are independent.}\end{equation}  The diffusion matrix and drift satisfy a restricted isotropy condition.  For every orthogonal transformation $r:\mathbb{R}^d\rightarrow\mathbb{R}^d$ which preserves the coordinate axes, for every $x\in\mathbb{R}^d$, \begin{equation}\label{isotropy} (b(rx,\omega),A(rx,\omega))\;\;\;\textrm{and}\;\;\;(rb(x,\omega),rA(x,\omega)r^t)\;\;\;\textrm{have the same law.}\end{equation}  And, finally, the diffusion matrix and drift are a small perturbation of the Laplacian.  There exists $\eta_0>0$, to later be chosen small, such that, for all $y\in\mathbb{R}^d$ and $\omega\in\Omega$, \begin{equation}\label{perturbation} \abs{b(y,\omega)}\leq\eta_0\;\;\textrm{and}\;\;\abs{A(y,\omega)-I}\leq \eta_0.\end{equation}

To avoid cumbersome statements in what follows, we introduce a steady assumption.  \begin{equation}\label{steady}\textrm{Assume}\;(\ref{sigmaalgebra}),(\ref{transgroup}), (\ref{stationary}), (\ref{bounded}), (\ref{Lipschitz}), (\ref{elliptic}), (\ref{finitedep}), (\ref{isotropy})\;\textrm{and}\;(\ref{perturbation}).\end{equation}

The collection of assumptions (\ref{transgroup}), (\ref{stationary}), (\ref{bounded}), (\ref{Lipschitz}) and (\ref{elliptic}) guarantee the well-posedness of the martingale problem set on $\mathbb{R}^d$, for each $\omega\in\Omega$ and $x\in\mathbb{R}^d$, associated to to the generator $$\frac{1}{2}\sum_{i,j=1}^da_{ij}(y,\omega)\frac{\partial^2}{\partial_{y_i}\partial_{y_j}}-\sum_{i=1}^db_i(y,\omega)\frac{\partial}{\partial_{y_i}},$$ see \cite{SV}.  We write $P_{x,\omega}$ and $E_{x,\omega}$ for the corresponding probability measure and expectation on the space of continuous paths $\C([0,\infty);\mathbb{R}^d)$ and remark that, almost surely with respect to $P_{x,\omega}$, paths $X_t\in\C([0,\infty);\mathbb{R}^d)$ satisfy the stochastic differential equation \begin{equation}\label{sde}\left\{\begin{array}{l} dX_t=-b(X_t,\omega)dt+\sigma(X_t,\omega)dB_t, \\ X_0=x,\end{array}\right.\end{equation} for $A(y,\omega)=\sigma(y,\omega)\sigma(y,\omega)^t$, and for $B_t$ a standard Brownian motion under $P_{x,\omega}$ with respect to the canonical right-continuous filtration on $\C([0,\infty);\mathbb{R}^d)$.

We write $\mathbb{P}_x=\mathbb{P}\times P_{x,\omega}$ and $\mathbb{E}_x=\mathbb{E}\times E_{x,\omega}$ for the corresponding semi-direct product measure and expectation on $\Omega\times\C([0,\infty);\mathbb{R}^d)$.  The annealed law $\mathbb{P}_x$ inherits the translation invariance and restricted rotational invariance implied by (\ref{stationary}) and (\ref{isotropy}).  In particular, for all $x,y\in\mathbb{R}^d$, \begin{equation}\label{annealed} \mathbb{E}_{x+y}(X_t)=\mathbb{E}_y(x+X_t)=x+\mathbb{E}_y(X_t),\end{equation} and, for all orthogonal transformations $r$ preserving the coordinate axis and $x\in\mathbb{R}^d$, \begin{equation}\label{annealed1} \mathbb{E}_{x}(rX_t)=\mathbb{E}_{rx}(X_t).\end{equation}   This stands in contrast to the quenched laws $P_{x,\omega}$, for which no invariance properties can be expected to hold, in general.

\subsection{A Review of \cite{SZ}}  In this section, we review the aspects of \cite{SZ} most relevant to our arguments.  Observe that this summary is by no means complete, as considerably more was achieved in their paper than we mention here.

We are interested in the long term behavior of the equation, for a fixed, H\"older continuous function $f:\mathbb{R}^d\rightarrow\mathbb{R}$, \begin{equation}\label{review_eq}\left\{\begin{array}{ll} u_t-\frac{1}{2}\tr(A(x,\omega)D^2u)+b(x,\omega)\cdot Du=0 & \textrm{on}\;\;\mathbb{R}^d\times(0,\infty), \\ u=f(x) & \textrm{on}\;\;\mathbb{R}^d\times\left\{0\right\}.\end{array}\right.\end{equation} This is essentially achieved by comparing the solutions of (\ref{review_eq}) to the solution of the deterministic problem, for $\overline{\alpha}>0$ identified in Theorem \ref{effectivediffusivity}, \begin{equation}\label{review_hom}\left\{\begin{array}{ll} \overline{u}_t-\frac{\overline{\alpha}}{2}\Delta \overline{u}=0 & \textrm{on}\;\;\mathbb{R}^d\times(0,\infty), \\ \overline{u}=f(x) & \textrm{on}\;\;\mathbb{R}^d\times\left\{0\right\},\end{array}\right.\end{equation} along an increasing sequence of length and time scales.

The constant $\overline{\alpha}$ determining (\ref{review_hom}) is identified in \cite{SZ} through a process we describe after introducing some notation.  Fix the dimension \begin{equation}\label{dimension} d\geq 3, \end{equation} and fix a H\"older exponent \begin{equation}\label{Holderexponent} \beta\in\left(0,\frac{1}{2}\right]\;\;\textrm{and a constant}\;\;a\in \left(0,\frac{\beta}{1000d}\right]. \end{equation}

Let $L_0$ be a large integer multiple of five.  For each $n\geq 0$, inductively define \begin{equation}\label{L} \ell_n=5\left[\frac{L_n^a}{5}\right]\;\;\textrm{and}\;\;L_{n+1}=\ell_n L_n, \end{equation} so that, for $L_0$ sufficiently large, we have $\frac{1}{2}L_n^{1+a}\leq L_{n+1}\leq 2L_n^{1+a}$.  For each $n\geq 0$, for $c_0>0$, let \begin{equation}\label{kappa} \kappa_n=\exp(c_0(\log\log(L_n))^2)\;\;\textrm{and}\;\;\tilde{\kappa}_n=\exp(2c_0(\log\log(L_n))^2),\end{equation} where we remark that, as $n$ tends to infinity, $\kappa_n$ is eventually dominated by every positive power of $L_n$.  Furthermore, define, for each $n\geq 0$, \begin{equation}\label{D} D_n=L_n\kappa_n\;\;\textrm{and}\;\;\tilde{D}_n=L_n\tilde{\kappa}_n.\end{equation}  We choose $L_0$ sufficiently large such that, for each $n\geq 0$, \begin{equation}\label{D_1} L_n<D_n< \tilde{D}_n< L_{n+1},\;\;4\tilde{\kappa}_n<\tilde{\kappa}_{n+1}\;\;\textrm{and}\;\; 3\tilde{D}_{n+1} < L_{n+1}^2.\end{equation}

The following constants enter into the probabilistic statements below.  Fix $m_0\geq 2$ satisfying \begin{equation}\label{m0} (1+a)^{m_0-2}\leq 100<(1+a)^{m_0-1}, \end{equation}  and $\delta>0$ and $M_0>0$ satisfying \begin{equation}\label{delta} \delta=\frac{5}{32}\beta\;\;\textrm{and}\;\;M_0\geq100d(1+a)^{m_0+2}.\end{equation}  In the arguments to follow, we will use the fact that $\delta$ and $M_0$ are sufficiently larger than $a$.

We now describe the identification of $\overline{\alpha}$.  Recall, for each $x\in\mathbb{R}^d$ and $\omega\in\Omega$, the quenched law $P_{x,\omega}$ on $\C([0,\infty);\mathbb{R}^d)$ and, for each $x\in\mathbb{R}^d$, the annealed law $\mathbb{P}_x$ on $\Omega\times\C([0,\infty);\mathbb{R}^d)$.  The constant $\overline{\alpha}$ is effectively identified as the limit of the effective diffusivities, in average, of the ensemble of equations (\ref{review_eq}) along the sequence of time steps $L_n^2$.  However, so as to apply the finite range dependence, see (\ref{finitedep}), the stopping time \begin{equation}\label{stopping} T_n=\inf\left\{\;s\geq0\;|\;\abs{X_s-X_0}\geq\tilde{D}_n\;\right\}\end{equation} is introduced, for each $n\geq 0$, and the approximate effective diffusivity of ensemble (\ref{review_eq}) is defined as \begin{equation}\label{alphan} \alpha_n=\frac{1}{dL_n^2}\mathbb{E}_0[\abs{X_{T_n\wedge L_n^2}}^2].\end{equation}  The following theorem describes the control and convergence of the $\alpha_n$ to $\overline{\alpha}$.

\begin{thm}\label{effectivediffusivity} Assume (\ref{steady}).  There exists $L_0$ and $c_0$ sufficiently large and $\eta_0>0$ sufficiently small such that, for all $n\geq 0$, $$\frac{1}{2\nu}\leq \alpha_n\leq 2\nu\;\;\textrm{and}\;\;\abs{\alpha_{n+1}-\alpha_n}\leq L_n^{-(1+\frac{9}{10})\delta},$$  which implies the existence of $\overline{\alpha}>0$ satisfying $$\frac{1}{2\nu}\leq \overline{\alpha}\leq 2\nu\;\;\textrm{and}\;\;\lim_{n\rightarrow\infty}\alpha_n=\overline{\alpha}.$$\end{thm}

We now describe the comparison between solutions of (\ref{review_eq}) and (\ref{review_hom}).  First, we compare solutions of (\ref{review_eq}), for each $n\geq 0$, at time $L_n^2$, with respect to a H\"{o}lder norm at scale $L_n$, to solutions of the deterministic problem \begin{equation}\label{review_approximate}\left\{\begin{array}{ll} u_{n,t}-\frac{\alpha_n}{2}\Delta u_n=0 & \textrm{on}\;\;\mathbb{R}^d\times(0,\infty), \\ u_{n,t}=f(x) & \textrm{on}\;\;\mathbb{R}^d\times\left\{0\right\}.\end{array}\right. \end{equation}  To do so, we introduce, for each $n\geq 0$, the rescaled H\"older norm \begin{equation}\label{levelholder} \abs{u_0}_n=\sup_{x\in\mathbb{R}^d}\abs{u_0(x)}+L_n^\beta\sup_{x\neq y}\frac{\abs{u_0(x)-u_0(y)}}{\abs{x-y}^\beta}.\end{equation}

We will obtain a localized control of the difference between solutions of (\ref{review_eq}) and (\ref{review_approximate}) at time $L_n^2$.  This localization is obtained via a cutoff function.  For each $v>0$, let\begin{equation}\label{cutoff} \chi(y)=1\wedge(2-\abs{y})_+\;\;\textrm{and}\;\;\chi_{v}(y)=\chi\left(\frac{y}{v}\right), \end{equation} and define, for each $x\in\mathbb{R}^d$ and $n\geq 0$, \begin{equation}\label{cutoff1}  \chi_{n,x}(y)=\chi_{30\sqrt{d}L_n}(y-x).\end{equation}  The following result then describes the desired comparison between solutions of (\ref{review_eq}) and (\ref{review_approximate}), at time $L_n^2$, for H\"older continuous initial data.

\begin{con}\label{Holder}  Fix $x\in\mathbb{R}^d$, $\omega\in\Omega$ and $n\geq 0$.  Let $u$ and $u_n$ respectively denote the solutions of (\ref{review_eq}) and (\ref{review_approximate}) corresponding to initial data $f\in\C^{0,\beta}(\mathbb{R}^d)$.  We have $$\abs{\chi_{n,x}(y)\left(u(y,L_n^2)-u_n(y,L_n^2)\right)}_n\leq L_n^{-\delta}\abs{f}_n.$$\end{con}

Notice that this control depends upon $x\in\mathbb{R}^d$, $\omega\in\Omega$ and $n\geq 0$.  It is not true, in general, that this type of contraction is available for all such triples $(x,\omega,n)$.  However, as described below, it is shown in \cite{SZ} that such controls are available for large $n$, with high probability, on a large portion of space.

The final control we will use concerns tail-estimates for the diffusion process.  We wish to control, under $P_{x,\omega}$, for $X_t\in\C([0,\infty);\mathbb{R}^d)$, the probability that \begin{equation}\label{star} X_t^*=\max_{0\leq s\leq t}\abs{X_s-X_0}\end{equation} is large with respect to the time elapsed.  The desired control contained in the following proposition is similar to the standard exponential estimates for Brownian motion for large length scales.

\begin{con}\label{localization}  Fix $x\in\mathbb{R}^d$, $\omega\in\Omega$ and $n\geq 0$.  For each $v\geq D_n$, for all $\abs{y-x}\leq 30\sqrt{d}L_n$, $$P_{y,\omega}(X^*_{L_n^2}\geq v)\leq \exp\left(-\frac{v}{D_n}\right).$$\end{con}

As with Control \ref{Holder}, this control depends upon $x\in\mathbb{R}^d$, $\omega\in\Omega$ and $n\geq 0$.  It is not true, in general, that this type of localization control is available for all such triples $(x,\omega,n)$, but it is shown in \cite{SZ} that such controls are available for large $n$, with high probability, on a large portion of space.

We now introduce the primary probabilistic statement concerning Controls \ref{Holder} and \ref{localization}.  Notice that the event defined below does not include the control of traps described in \cite{SZ}, which play in important role in propagating Control \ref{Holder} in their arguments.  Since we simply use the H\"older control there obtained, we do not require a further use of their control of traps.

Consider, for each $x\in\mathbb{R}^d$, the event \begin{equation}\label{mainevent} \mathcal{B}_n(x)=\left\{\;\omega\in\Omega\;|\;\textrm{Controls \ref{Holder} and \ref{localization} hold for the triple}\;(x,\omega,n).\;\right\}.\end{equation}  Notice that, in view of (\ref{stationary}), for all $x\in\mathbb{R}^d$ and $n\geq 0$, \begin{equation}\label{mainevent1}\mathbb{P}(B_n(x))=\mathbb{P}(B_n(0)).\end{equation} It is therefore shown that the probability of the compliment of $B_n(0)$ approaches zero as $n$ tends to infinity.

\begin{thm}\label{induction}  Assume (\ref{steady}).  There exist $L_0$ and $c_0$ sufficiently large and $\eta_0>0$ sufficiently small such that, for each $n\geq 0$, $$\mathbb{P}\left(\Omega\setminus B_n(0)\right)\leq L_n^{-M_0}.$$\end{thm}

We henceforth fix the constants $L_0$, $c_0$ and $\eta_0$ appearing above.  \begin{equation}\label{constants} \textrm{Fix constants}\;L_0, c_0\;\textrm{and}\;\eta_0\;\textrm{satisfying (\ref{D_1}) and the hypothesis of Theorems \ref{effectivediffusivity} and \ref{induction}.}\end{equation}

We conclude this section with a few basic observations concerning Control \ref{Holder}, Control \ref{localization} and the H\"older norms introduced in (\ref{levelholder}).  Since Control \ref{Holder} cannot be expected to hold globally in space, it will be frequently necessary to introduce cutoff functions of the type appearing in (\ref{cutoff}).  The primary purpose of Control \ref{localization} is to bound the error we introduce, as seen in the following proposition.

\begin{prop}\label{local11}  Assume (\ref{steady}) and (\ref{constants}).  Fix $x\in\mathbb{R}^d$, $\omega\in\Omega$ and $n\geq 0$ and suppose that Control \ref{localization} is satisfied for the triple $(x,\omega,n)$.  For $f\in L^\infty(\mathbb{R}^d)$ satisfying $$\d\left(\Supp(f), B_{30\sqrt{d}L_n}(x)\right)\geq D_n+30\sqrt{d}L_n,$$ let $u(y,t)$ satisfy (\ref{review_eq}) with initial data $f(y)$.  Then, for each $\abs{y-x}\leq 30\sqrt{d}L_n$, $$\abs{u(y,L_n^2)}\leq \exp\left(-\frac{\d(\Supp(f),y)}{D_n}\right)\norm{f}_{L^\infty(\mathbb{R}^d)}.$$\end{prop}

\begin{proof}  The proof is immediate from the representation formula for the solution.  We have, for each $y\in\mathbb{R}^d$, $$u(y,L_n^2)=P_{y,\omega}\left(f(X_{L_n^2})\right).$$  Therefore, $$\abs{u(y,L_n^2)}\leq P_{y,\omega}\left(X_{L_n^2}^*\geq \d(\Supp(f),y)\right)\norm{f}_{L^\infty(\mathbb{R}^d)}.$$  Since $\d(\Supp(f), B_{30\sqrt{d}L_n}(x))\geq D_n+30\sqrt{d}L_n$, and since Control \ref{localization} is satisfied for the triple $(x,\omega,n)$, this implies that, for all $\abs{y-x}\leq 30\sqrt{d}L_n$, $$\abs{u(y,L_n^2)}\leq \exp\left(-\frac{\d(\Supp(f),y)}{D_n}\right)\norm{f}_{L^\infty(\mathbb{R}^d)},$$ which completes the argument.  \end{proof}

The following two elementary propositions will be used to extend Control \ref{Holder} to a larger portion of space.  The first is an elementary and well-known fact concerning the product of H\"older continuous functions.

\begin{prop}\label{prelim_product}  For each $n\geq 0$, for every $f,g\in\C^{0,\beta}(\mathbb{R}^d)$, $$\abs{fg}_n\leq \abs{f}_n\abs{g}_n.$$\end{prop}

\begin{proof}  Fix $n\geq 0$ and $f,g\in\C^{0,\beta}(\mathbb{R}^d)$.  For every $x,y\in\mathbb{R}^d$, the triangle inequality implies $$\abs{f(x)g(x)-f(y)g(y)}\leq \abs{f(x)}\abs{g(x)-g(y)}+\abs{g(y)}\abs{f(x)-f(y)}.$$  Therefore, $$\sup_{x\neq y}L_n^\beta\frac{\abs{f(x)g(x)-f(y)g(y)}}{{\abs{x-y}^\beta}}\leq \norm{f}_{L^\infty(\mathbb{R}^d)}\sup_{x\neq y}L_n^\beta\frac{\abs{g(x)-g(y)}}{\abs{x-y}^\beta}+\norm{g}_{L^\infty(\mathbb{R}^d)}\sup_{x\neq y}L_n^\beta\frac{\abs{f(x)-f(y)}}{\abs{x-y}^\beta}.$$  And, since $$\norm{fg}_{L^\infty(\mathbb{R}^d)}\leq \norm{f}_{L^\infty(\mathbb{R}^d)}\norm{g}_{L^\infty(\mathbb{R}^d)},$$ we conclude that $$\abs{fg}_n\leq\norm{f}_{L^\infty(\mathbb{R}^d)}\abs{g}_n+\norm{g}_{L^\infty(\mathbb{R}^d)}\sup_{x\neq y}L_n^\beta\frac{\abs{f(x)-f(y)}}{\abs{x-y}^\beta}\leq \abs{f}_n\abs{g}_n,$$ which completes the argument.  \end{proof}

The final proposition will play the most important role in extending Control \ref{Holder}.  The only observation is that the H\"older norms introduced in (\ref{levelholder}) occur at the length scale $L_n$.  Therefore, a function agreeing locally with H\"older continuous functions on scale $L_n$ must itself be globally H\"older continuous.

\begin{prop}\label{prelim_extension}  Let $I$ be an arbitrary index and $n\geq 0$.  If $f:\mathbb{R}^d\rightarrow\mathbb{R}$ and $\left\{g_i:\mathbb{R}^d\rightarrow\mathbb{R}\right\}_{i\in I}$ are such that, for a collection $\left\{x_i\right\}_{i\in I}\subset \mathbb{R}^d$,\begin{equation}\label{Holder2}f=g_i\;\;\textrm{on}\;\;B(x_i, 20\sqrt{d}L_n) \;\;\textrm{and}\;\; \Supp(f)\subset \bigcup_{i\in I}B(x_i,10\sqrt{d}L_n),\end{equation} then $$\abs{f}_n\leq 3\sup_{i\in I}\abs{g_i}_n.$$\end{prop}

\begin{proof}  In view of (\ref{Holder2}), for each $x\in\mathbb{R}^d$ there exists $j\in I$ such that $f(x)=g_j(x)$.  Therefore, \begin{equation}\label{Holder3} \abs{f(x)}=\abs{g_j(x)}\leq \sup_{i\in\mathbb{I}}\abs{g_i}_n.\end{equation}  It remains to bound the H\"{o}lder semi-norm.

If $x,y\in\mathbb{R}^d$ satisfy $\abs{x-y}\geq L_n$, in view of (\ref{Holder2}), for $j,k\in I$ satisfying $f(x)=g_j(x)$ and $f(y)\leq g_k(y)$, \begin{equation}\label{Holder4} L_n^\beta\frac{\abs{f(x)-f(y)}}{\abs{x-y}^\beta}\leq \abs{g_j(x)-g_k(y)}\leq 2\sup_{i\in I}\abs{g_i}_n.\end{equation}  If $\abs{x-y}<L_n$, in view of (\ref{Holder2}), there exists $j\in I$ such that $x,y\in B(x_j,20\sqrt{d}L_n).$  Therefore, for this $j\in I$, \begin{equation}\label{Holder5} L_n^\beta\frac{\abs{f(x)-f(y)}}{\abs{x-y}^\beta}=L_n^\beta\frac{\abs{g_j(x)-g_j(y)}}{\abs{x-y}^\beta}\leq \abs{g_j}_n\leq\sup_{i\in I}\abs{g_i}_n.\end{equation}  The claim follows by combining (\ref{Holder3}), (\ref{Holder4}) and (\ref{Holder5}).  \end{proof}

\section{The Identification of the Invariant Measure}

In order to identify the invariant measure, we will analyze the long term behavior of the solution $u:\mathbb{R}^d\times[0,\infty)\times\Omega\rightarrow\mathbb{R}$ satisfying \begin{equation}\label{o_eq} \left\{\begin{array}{ll} u_t-\frac{1}{2}\tr(A(x,\omega)D^2u)+b(x,\omega)\cdot Du=0 & \textrm{on}\;\;\mathbb{R}^d\times(0,\infty), \\ u=f(x,\omega) & \textrm{on}\;\;\mathbb{R}^d\times\left\{0\right\}.\end{array}\right.\end{equation} Therefore, to simplify the notation in what follows, we write, for each $s\geq 0$ and $\omega\in\Omega$, $$R_sf(x,\omega)=u(x,s,\omega),$$ for $u(x,s,\omega)$ satisfying (\ref{o_eq}) with initial data $f(y,\omega)$.

We will be particularly interested in translations of functions $\tilde{f}\in L^\infty(\Omega)$ with respect to the translation group $\left\{\tau_x\right\}_{x\in\mathbb{R}^d}$ governing the stationarity of the coefficients, and therefore assume in many of the propositions to follow that a function $f:\mathbb{R}^d\times\Omega\rightarrow\mathbb{R}$ is stationary with respect to $\left\{\tau_x\right\}_{x\in\mathbb{R}^d}$.  Precisely, for each $x,y\in\mathbb{R}^d$ and $\omega\in\Omega$, \begin{equation}\label{o_stationary} f(x+y,\omega)=f(x,\tau_y,\omega).\end{equation}

For every $f\in L^\infty(\mathbb{R}^d\times\Omega)$ satisfying (\ref{o_stationary}), we identify a deterministic constant $\overline{\pi}(f)\in\mathbb{R}$ which is effectively identified as the limit of the sequence defined, for each $n\geq 0$, by \begin{equation}\label{o_pi} \mathbb{E}\left(R_{L_n^2}f(0,\omega)\right).\end{equation}  And, for ${\bf{1}}_E:\Omega\rightarrow \mathbb{R}$ the indicator function of a measurable subset $E\in\mathcal{F}$, by taking $f_E(x,\omega)={\bf{1}}_E(\tau_x\omega)$, we define a measure $\pi:\mathcal{F}\rightarrow\mathbb{R}$ on $(\Omega,\mathcal{F})$ by the rule \begin{equation}\label{o_pi_1} \pi(E)=\overline{\pi}(f_E).\end{equation}  We will prove that $\pi$ is a probability measure on $(\Omega,\mathcal{F})$ which is absolutely continuous with respect to $\mathbb{P}$.  And, for every $f\in L^\infty(\mathbb{R}^d\times\Omega)$ satisfying (\ref{o_stationary}), $$\overline{\pi}(f)=\int_{\Omega}f(0,\omega)\;d\pi.$$

The following two propositions describe the basic existence and regularity results concerning equation (\ref{o_eq}) for bounded and stationary initial data.

\begin{prop}\label{o_sol}  Assume (\ref{steady}).  For each $\omega\in\Omega$ and $f\in L^\infty(\mathbb{R}^d\times\Omega)$ there exists a unique solution $u(x,t,\omega):\mathbb{R}^d\times[0,\infty)\times\Omega\rightarrow\mathbb{R}$ of (\ref{o_eq}) satisfying, for each $T>0$ and $\omega\in\Omega$, $u(x,t,\omega)\in\BUC(\mathbb{R}^d\times[0,T])$ with, for each $\omega\in\Omega$, $$\norm{u(x,t,\omega)}_{L^\infty(\mathbb{R}^d\times[0,\infty))}\leq\norm{f(x,\omega)}_{L^\infty(\mathbb{R}^d)}.$$  Furthermore, if $f(x,\omega)$ satisfies (\ref{o_stationary}), then for each $t\geq 0$, the map $u(x,t,\omega):\mathbb{R}^d\times\Omega\rightarrow\mathbb{R}^d$ is stationary.  Precisely, for each $x,y\in\mathbb{R}^d$, $t\geq 0$ and $\omega\in\Omega$, $$u(x,t,\tau_y\omega)=u(x+y,t,\omega).$$ \end{prop}

\begin{proof}  The existence and uniqueness of a solution to (\ref{o_eq}) satisfying the above estimates, for each $\omega\in\Omega$, is an elementary consequence of (\ref{bounded}), (\ref{Lipschitz}) and $f\in L^\infty(\mathbb{R}^d\times\Omega)$.  See, for instance, Friedman \cite{Fr}.  The stationarity is a consequence of (\ref{o_stationary}) and the uniqueness since, for each $\omega\in\Omega$, both $u(x,t,\tau_y\omega)$ and $u(x+y,t,\omega)$ satisfy (\ref{o_eq}) for $\tau_y\omega$.  \end{proof}

\begin{prop}\label{o_weak}  Assume (\ref{steady}).  For each $\omega\in\Omega$, $t\geq 1$ and $g\in L^\infty(\mathbb{R}^d)$, for $C>0$ independent of $\omega\in\Omega$ and $t\geq 1$, $$\norm{R_tg(x,\omega)}_{C^{0,\beta}(\mathbb{R}^d)}\leq C\norm{g}_{L^\infty(\mathbb{R}^d)}.$$\end{prop}

\begin{proof}  Fix $\omega\in\Omega$ and $g\in L^\infty(\mathbb{R}^d)$.  Recall that, for each $t\geq 0$ and $x\in\mathbb{R}^d$, see \cite{Fr}, \begin{equation}\label{o_weak_1}R_tg(x,\omega)=P_{x,\omega}\left(g(X_t)\right)=\int_{\mathbb{R}^d}p(x,t,y,\omega)g(y)\;dy,\end{equation} for $p(x,t,y,\omega):\mathbb{R}^d\times(0,\infty)\times\mathbb{R}^d\rightarrow\mathbb{R}$ satisfying, for each $0<t\leq 1$, for $C>0$ and $c>0$ independent of $\omega$,\begin{equation}\label{o_weak_2} \abs{p(x,t,y,\omega)}\leq Ct^{-d/2}e^{-c\abs{x-y}^2/t}\;\;\textrm{and}\;\;\abs{D_xp(x,t,y,\omega)}\leq Ct^{-(d+1)/2}e^{-c\abs{x-y}^2/t}.\end{equation}

First, we observe that for each $x\in\mathbb{R}^d$ and $t\geq 0$, using (\ref{o_weak_1}), \begin{equation}\label{o_weak_3} \abs{R_tg(x,\omega)}\leq \norm{g}_{L^\infty(\mathbb{R}^d)}.\end{equation}   It remains to bound the H\"older semi-norm.

Whenever $x,y\in\mathbb{R}^d$ satisfy $\abs{x-y}\geq 1$, \begin{equation}\label{o_weak_4} \abs{R_1g(x,\omega)-R_1g(y,\omega)}\leq 2\norm{g}_{L^\infty(\mathbb{R}^d)}\leq 2\abs{x-y}^\beta\norm{g}_{L^\infty(\mathbb{R}^d)}.\end{equation}  And, whenever $x,y\in\mathbb{R}^d$ satisfy $\abs{x-y}<1$, in view of (\ref{o_weak_1}) and (\ref{o_weak_2}), for $C>0$ independent of $\omega\in\Omega$, \begin{equation}\label{o_weak_5}\abs{R_1g(x,\omega)-R_1g(y,\omega)}\leq C\abs{x-y}\norm{g}_{L^\infty(\mathbb{R}^d)}\leq C\abs{x-y}^\beta\norm{g}_{L^\infty(\mathbb{R}^d)}.\end{equation}  Therefore, for each $x,y\in\mathbb{R}^d$ and $t\geq 1$, using (\ref{o_weak_3}), (\ref{o_weak_4}) and (\ref{o_weak_5}), \begin{multline}\label{o_weak_6} \abs{R_tg(x,\omega)-R_tg(y,\omega)}=\abs{R_{t-1}(R_1g(x,\omega)-R_1g(y,\omega))} \\ \leq \sup_{x,y\in\mathbb{R}^d}\abs{R_1g(x,\omega)-R_1g(y,\omega)}\leq C\abs{x-y}^\beta\norm{g}_{L^\infty(\mathbb{R}^d)}.\end{multline} The claim follows from (\ref{o_weak_3}), (\ref{o_weak_4}) and (\ref{o_weak_6}), since $\omega\in\Omega$ and $g\in L^\infty(\mathbb{R}^d)$ were arbitrary.  \end{proof}

Before proceeding with the proof, it is convenient to introduce some useful notation.  We write, for each $n\geq 0$ and $f\in\C^{0,\beta}(\mathbb{R}^d)$, \begin{equation}\label{o_op} R_nf(x,\omega)=u(x,L_n^2),\end{equation} for $u(x,t)$ satisfying $$\left\{\begin{array}{ll} u_t-\frac{1}{2}\tr(A(x,\omega)D^2u)+b(x,\omega)\cdot Du=0 & \textrm{on}\;\;\mathbb{R}^d\times(0,\infty), \\ u=f(x) & \textrm{on}\;\;\mathbb{R}^d\times\left\{0\right\}.\end{array}\right.$$  Similarly, we define, $n\geq 0$ and $f\in\C^{0,\beta}(\mathbb{R}^d)$, \begin{equation}\label{o_hop} \overline{R}_nf(x)=\overline{u}_n(x,t),\end{equation} for $\overline{u}_n(x,t)$ satisfying $$\left\{\begin{array}{ll} \overline{u}_{n,t}-\frac{\alpha_n}{2}\Delta\overline{u}_n=0 & \textrm{on}\;\;\mathbb{R}^d\times(0,\infty), \\ \overline{u}_n=f(x) & \textrm{on}\;\;\mathbb{R}^d\times\left\{0\right\}.\end{array}\right.$$  And, finally, for each $n\geq 0$ and $f\in\C^{0,\beta}(\mathbb{R}^d)$, we define \begin{equation}\label{o_dif} S_nf(x,\omega)=R_nf(x,\omega)-\overline{R}_nf(x).\end{equation}  This allows us to restate Control \ref{Holder} in the following equivalent way, where we recall from (\ref{cutoff1}), for each $x\in\mathbb{R}^d$ and $n\geq 0$, the cutoff function $\chi_{n,x}$.

\begin{con}\label{o_Holder}  Fix $x\in\mathbb{R}^d$, $\omega\in\Omega$ and $n\geq 0$.  For each $f\in\C^{0,\beta}(\mathbb{R}^d)$, $$\abs{\chi_{n,x}S_nf}_n\leq L_n^{-\delta}\abs{f}_n.$$\end{con}

We now make two elementary observations concerning the interaction of the heat kernels $\overline{R}_n$ introduced in (\ref{o_hop}) and the scaled H\"older norms introduced in (\ref{levelholder}), and an observation concerning the localization properties of the kernels $\overline{R}_n$.  Notice that, in the following proposition, we make use of Theorem \ref{effectivediffusivity}, which in particular provides a lower bound for the $\alpha_n$.  This lower bound ensures that the kernels $\overline{R}_n$ provide a sufficient regularization, uniformly in $n\geq 0$, for our arguments to follow.

\begin{prop}\label{o_regular}  Assume (\ref{steady}) and (\ref{constants}).  There exists $C>0$ satisfying, for each $n\geq 0$ and $f\in L^\infty(\mathbb{R}^d)$, $$\abs{\overline{R}_nf}_n\leq C\norm{f}_{L^\infty(\mathbb{R}^d)}.$$\end{prop}

\begin{proof}  Fix $n\geq 0$ and $f\in L^\infty(\mathbb{R}^d)$.  In view of (\ref{o_hop}), for each $x\in\mathbb{R}^d$, $$\overline{R}_nf(x)=\int_{\mathbb{R}^d}(4\pi \alpha_n L_n^2)^{-d/2}e^{-\abs{x-y}^2/4\alpha_n L_n^2}f(y)\;dy.$$  Therefore, \begin{equation}\label{o_regular_1}\norm{\overline{R}_nf(x)}_{L^\infty(\mathbb{R}^d)}\leq\norm{f}_{L^\infty(\mathbb{R}^d)}.\end{equation}  It remains to bound the H\"older semi-norm.

For each $x\in\mathbb{R}^d$, $$D\overline{R}_nf(x)=\pi^{-d/2}(4\alpha_n L_n^2)^{-1/2}\int_{\mathbb{R}^d}\frac{x-y}{(4\alpha_n L_n^2)^{(d+1)/2}}e^{-\abs{x-y}^2/4\alpha_n L_n^2}f(y)\;dy.$$  Therefore, in view of Theorem \ref{effectivediffusivity}, for each $x\in\mathbb{R}^d$, for $C>0$ independent of $n\geq 0$ and $f\in L^\infty(\mathbb{R}^d)$, $$\abs{D\overline{R}_nf(x)}=\left|\pi^{-d/2}(4\alpha_n L_n^2)^{-1/2}\int_{\mathbb{R}^d} ye^{-\abs{y}^2}f\left(4\alpha_nL_n^2)^{1/2}y+x\right)\;dy\right|\leq CL_n^{-1}\norm{f}_{L^\infty(\mathbb{R}^d)}.$$  So, whenever $x,y\in\mathbb{R}^d$ satisfy $0<\abs{x-y}<L_n$, \begin{equation}\label{o_regular_2} L_n^\beta\frac{\abs{\overline{R}_nf(x)-\overline{R}_nf(y)}}{\abs{x-y}^\beta}\leq CL_n^{\beta-1}\norm{f}_{L^\infty(\mathbb{R}^d)}\abs{x-y}^{1-\beta}\leq \norm{f}_{L^\infty(\mathbb{R}^d)}.\end{equation}  And, in view of (\ref{o_regular_1}), if $\abs{x-y}\geq L_n$, \begin{equation}\label{o_regular_3}  L_n^\beta\frac{\abs{\overline{R}_nf(x)-\overline{R}_nf(y)}}{\abs{x-y}^\beta}\leq 2\norm{f}_{L^\infty(\mathbb{R}^d)}.\end{equation}  The claim follows from (\ref{o_regular_1}), (\ref{o_regular_2}) and (\ref{o_regular_3}).  \end{proof}

The following observation is elementary and well-known.  The kernels $\overline{R}_n$ preserve H\"older continuous initial data.

\begin{prop}\label{o_contract}  For each $n\geq 0$ and $f\in\C^{0,\beta}(\mathbb{R}^d)$, $$\abs{\overline{R}_nf}_n\leq\abs{f}_n.$$\end{prop}

\begin{proof}  Fix $n\geq 0$ and $f\in\C^{0,\beta}(\mathbb{R}^d)$.  For each $x\in\mathbb{R}^d$, $$\overline{R}_nf(x)=\int_{\mathbb{R}^d}(4\pi \alpha_n L_n^2)^{-d/2}e^{-\abs{y}^2/4\alpha_n L_n^2}f(y+x)\;dy.$$  Therefore, \begin{equation}\label{0_contract_1}\norm{\overline{R}_nf}_{L^\infty(\mathbb{R}^d)}\leq\norm{f}_{L^\infty(\mathbb{R}^d)}.\end{equation}  It remains to bound the H\"older semi-norm.

Fix elements $y\neq z$ of $\mathbb{R}^d$.  Then, $$\abs{\overline{R}_nf(y)-\overline{R}_nf(z)}=\left|\int_{\mathbb{R}^d}(4\pi \alpha_n L_n^2)^{-d/2}e^{-\abs{y}^2/4\alpha_n L_n^2}\left(f(y+x)-f(z+x)\right)\;dy,\right| $$ and, therefore, \begin{equation}\label{0_contract_2}\sup_{y\neq z}\frac{\abs{\overline{R}_nf(y)-\overline{R}_nf(z)}}{\abs{y-z}^\beta}\leq \sup_{y\neq z}\frac{\abs{f(y)-f(z)}}{\abs{y-z}^\beta}.\end{equation}  The claim follows from (\ref{0_contract_1}) and (\ref{0_contract_2}).  \end{proof}

Finally, the following proposition describes the localization properties of the kernels $\overline{R}_n$.  Here, notice again the role of Theorem \ref{effectivediffusivity} and recall the cutoff function introduced in (\ref{cutoff}).

\begin{prop}\label{o_localize}  Assume (\ref{steady}) and (\ref{constants}).  There exits $C=C(d)>0$ and $c>0$ independent of $n$ such that, for each $f\in L^\infty(\mathbb{R}^d)$, $$\abs{\overline{R}_n(1-\chi_{\tilde{D}_n})f(0)}\leq Ce^{-c\tilde{\kappa}_n^2}\norm{f}_{L^\infty(\mathbb{R}^d)}.$$  \end{prop}
 
\begin{proof}  Fix $n\geq 0$.  Then, for $C=C(d)>0$, \begin{multline*}\abs{\overline{R}_n(1-\chi_{\tilde{D}_n})f(0)}\leq \int_{\mathbb{R}^d\setminus B_{\tilde{D}_n}}(4\pi \alpha_n L_n^2)^{-d/2}e^{-\abs{x-y}^2/4\alpha_n L_n^2}f(y)\;dy \\ \leq C\norm{f}_{L^\infty(\mathbb{R}^d)}\int_{\tilde{D}_n/2\sqrt{\alpha_n}L_n}^\infty re^{-r^2}\;dr.\end{multline*}  Therefore, using Theorem \ref{effectivediffusivity}, there exists $c>0$ independent of $n$ such that, for $C=C(d)>0$, $$\abs{\overline{R}_n(1-\chi_{\tilde{D}_n})f(0)}\leq Ce^{-\tilde{\kappa}_n^2/4\alpha_n}\norm{f}_{L^\infty(\mathbb{R}^d)}\leq Ce^{-\tilde{c\kappa}_n^2}\norm{f}_{L^\infty(\mathbb{R}^d)} ,$$ which completes the argument.  \end{proof}

We are now prepared to proceed with the main argument.  In order to exploit the finite range dependence in what follows, see (\ref{finitedep}), we introduce localized versions of the kernels $R_n$.  Define, for each $n\geq 0$ and $\omega\in\Omega$, $$\tilde{R}_nf(x,\omega)=\tilde{u}(x,L_n^2,\omega),$$ for $\tilde{u}:\overline{B}_{6\tilde{D}_n}\times[0,\infty)\times\Omega\rightarrow\mathbb{R}^d$ satisfying \begin{equation}\label{o_localized} \left\{\begin{array}{ll} \tilde{u}_t-\frac{1}{2}\tr(A(y,\omega)D^2\tilde{u})+b(y,\omega)\cdot D\tilde{u}=0 & \textrm{on}\;\;B_{6\tilde{D}_n}(x)\times(0,\infty), \\ \tilde{u}=f(y,\omega) & \textrm{on}\;\;B_{6\tilde{D}_n}(x)\times\times\left\{0\right\}, \\ \tilde{u}=f(y,\omega) & \textrm{on}\;\;\partial B_{6\tilde{D}_n}(x)\times(0,\infty).\end{array}\right.\end{equation}  The following proposition describes the basic properties of the solutions to (\ref{o_localized}).

\begin{prop}\label{o_localsol}  Assume (\ref{steady}).  For each $x\in\mathbb{R}^d$ and $f\in L^\infty(\mathbb{R}^d\times\Omega)$ there exists a unique solution $\tilde{u}(y,t,\omega):\overline{B}_{6\tilde{D}_n}\times[0,\infty)\times\Omega\rightarrow\mathbb{R}$ of (\ref{o_localized}) satisfying, for each $T>0$, $x\in\mathbb{R}^d$ and $\omega\in\Omega$, $\tilde{u}(y,t,\omega)\in\BUC(\overline{B}_{6\tilde{D}_n}(x)\times[0,T])$ with, for each $\omega\in\Omega$, $$\norm{\tilde{u}(y,t,\omega)}_{L^\infty(\overline{B}_{6\tilde{D}_n}(x)\times[0,\infty))}\leq\norm{f(y,\omega)}_{L^\infty(\mathbb{R}^d\times\Omega)}.$$  Furthermore, if $f(x,\omega)$ satisfies (\ref{o_stationary}), then for each $n\geq 0$ and $k\geq 0$, the map $\left(\tilde{R}_n\right)^kf(x,\omega):\mathbb{R}^d\times\Omega\rightarrow\mathbb{R}^d$ is stationary.  Precisely, for each $x,y\in\mathbb{R}^d$ and $\omega\in\Omega$, $$\left(\tilde{R}_n\right)^kf(x,\tau_y\omega)=\left(\tilde{R}_n\right)^kf(x+y,\omega).$$\end{prop}

\begin{proof}  Fix $n\geq 0$ and $k\geq 0$.  The existence and uniqueness of a solution to (\ref{o_localized}) satisfying the above estimates, for each $\omega\in\Omega$, is an elementary consequence of (\ref{bounded}), (\ref{Lipschitz}) and $f\in L^\infty(\mathbb{R}^d\times\Omega)$.  See, for instance, \cite{Fr}.  The stationarity is a consequence of (\ref{stationary}) and the uniqueness since, for each $\omega\in\Omega$ and $x,y\in\mathbb{R}^d$, if $\tilde{u}(\cdot,\cdot,\omega)$ satisfies (\ref{o_localized}) corresponding to $\omega$ on $\overline{B}_{6\tilde{D}_n}(x+y)\times[0,\infty)$ then $\tilde{u}(\cdot+y,\cdot,\omega)$ satisfies (\ref{o_localized}) corresponding to $\tau_y\omega$ on $\overline{B}_{6\tilde{D}_n}(x)\times[0,\infty)$.  \end{proof}

We now obtain Controls \ref{localization} and \ref{o_Holder} on a large portion of space, with high probability.  Define, for each $n\geq 0$, $$\tilde{A}_n=\left\{\;\omega\in\Omega\;|\;\omega\in B_n(x)\;\textrm{for all}\;x\in L_n\mathbb{Z}^d\cap[-2L_{n+2}^2, 2L_{n+2}^2]^d.\;\right\},$$ and, for each $n\geq 0$, \begin{equation}\label{o_subset} A_n=\tilde{A_n}\cap \tilde{A}_{n+1}\cap \tilde{A}_{n+2}.\end{equation}  The following proposition provides, for each $n\geq 0$, a lower bound for the probability of $A_n$.

\begin{prop}\label{o_probability}  Assume (\ref{steady}) and (\ref{constants}).  For each $n\geq 0$, for $C>0$ independent of $n$, $$\mathbb{P}(\Omega\setminus A_n)\leq CL_n^{(2(1+a)^2-1)d-M_0}.$$\end{prop}

\begin{proof}  In view of (\ref{mainevent1}), for each $n\geq 0$, for $C>0$ independent of $n$, $$\mathbb{P}(\Omega\setminus \tilde{A}_n)\leq \sum_{x\in L_n\mathbb{Z}^d\cap[-2L_{n+2}^2, 2L_{n+2}^2]^d}\mathbb{P}\left(\Omega\setminus B_n(x)\right)\leq C\left(L_{n+2}^2/L_n\right)^d\mathbb{P}\left(\Omega\setminus B_n(0)\right).$$  Therefore, using Theorem \ref{induction}, for each $n\geq 0$, for $C>0$ independent of $n$, $$\mathbb{P}(\Omega\setminus \tilde{A}_n)\leq CL_n^{(2(1+a)^2-1)d-M_0}.$$  This implies that, for each $n\geq 0$, \begin{multline*}\mathbb{P}\left(\Omega\setminus A_n\right)\leq \mathbb{P}\left(\Omega\setminus \tilde{A}_n\right)+\mathbb{P}\left(\Omega\setminus\tilde{A}_{n+1}\right)+\mathbb{P}\left(\Omega\setminus \tilde{A}_{n+2}\right) \\ \leq C\left(L_n^{(2(1+a)^2-1)d-M_0}+L_{n+1}^{(2(1+a)^2-1)d-M_0}+L_{n+2}^{(2(1+a)^2-1)d-M_0}\right)\leq CL_n^{(2(1+a)^2-1)d-M_0},\end{multline*} which completes the argument.  \end{proof}

We remark that, in view of (\ref{Holderexponent}) and (\ref{delta}), the exponent $(2(1+a)^2-1)d-M_0<0$.  The following proposition provides the first step toward comparing $R_{n+1}f(x,\omega)$ to $R_nf(x,\omega)$.  We obtain this comparison on the subset $A_n$ defined in (\ref{o_subset}).

Notice that the estimates contained in the following proposition depend on the unscaled, $\beta$-H\"older norm of the initial data.  To identify the invariant measure, since this requires us to consider initial data $f\in L^\infty(\mathbb{R}^d\times\Omega)$, we will use Proposition \ref{o_weak} and apply the following result to $R_1f(x,\omega)$.

\begin{prop}\label{o_main}  Assume (\ref{steady}) and (\ref{constants}).  For each $n\geq 0$, $\omega\in A_n$, $1\leq k<\ell_{n+1}^2$ and $f\in C^{0,\beta}(\mathbb{R}^d)$, for $C>0$ independent of $n$, $$\sup_{x\in B_{4\sqrt{k}\tilde{D}_{n+1}}}\abs{\left(R_{n+1}\right)^kf(x,\omega)-\left(\overline{R}_n\right)^{k\ell_n^2-6}\left(\tilde{R}_n\right)^6f(x,\omega)}\leq CL_n^{\beta-7(\delta-5a)}\norm{f}_{C^{0,\beta}(\mathbb{R}^d)}.$$\end{prop}

\begin{proof}  Fix $n\geq 0$, $\omega\in A_n$, $1\leq k<\ell_{n+1}^2$ and $f\in C^{0,\beta}(\mathbb{R}^d)$.  In what follows, we suppress the dependence on $\omega\in\Omega$.  Notice that (\ref{L}), (\ref{kappa}) and (\ref{D}) imply that, since $1\leq k<\ell_{n+1}^2$, we have \begin{equation}\label{o_main_00} 4\sqrt{k}\tilde{D}_{n+1}<\tilde{D}_{n+2}.\end{equation}  Also, notice that (\ref{D_1}) implies that, in the definition of $A_n$, we have \begin{equation}\label{o_main_0} 3\tilde{D}_{n+2}< L_{n+2}^2.\end{equation}  And, for what follows, we recall that $$\left(R_{n+1}\right)^kf(x)=\left(R_n\right)^{k\ell_n^2}f(x).$$

Fix $x\in B_{4\sqrt{k}\tilde{D}_{n+1}}$ and define the cutoff function $\tilde{\chi}_{n,x}:\mathbb{R}^d\rightarrow\mathbb{R}^d$, recalling (\ref{cutoff}), \begin{equation}\label{o_main_cut}\tilde{\chi}_{n,x}(y)=\chi_{\tilde{D}_{n+2}}(y-x)\;\;\textrm{on}\;\;\mathbb{R}^d.\end{equation}  Since \begin{equation}\label{o_main_1} \norm{R_nf}_{L^\infty(\mathbb{R}^d)}\leq\norm{f}_{L^\infty(\mathbb{R}^d)},\end{equation} and since $x\in B_{4\sqrt{k}\tilde{D}_{n+1}}$ and $\omega\in A_n$, Control \ref{localization}, Proposition \ref{local11}, (\ref{o_main_00}) and (\ref{o_main_0}) imply that $$\abs{\left(R_n\right)^{k\ell_n^2}f(x)-\left(R_n\right)^{k\ell_n^2-1}\tilde{\chi}_{n,x}R_nf(x)}=\abs{\left(R_n\right)^{k\ell_n^2-1}(1-\tilde{\chi}_{n,x})R_nf(x)}\leq e^{-\kappa_{n+2}}\norm{f}_{L^\infty(\mathbb{R}^d)}.$$  Proceeding inductively, we conclude that \begin{equation}\label{o_main_2} \abs{\left(R_n\right)^{k\ell_n^2}f(x)-\left(\tilde{\chi}_{n,x}R_n\right)^{k\ell_n^2}f(x)}\leq k\ell_n^2e^{-\kappa_{n+2}}\norm{f}_{L^\infty(\mathbb{R}^d)}\leq \ell_{n+1}^2\ell_n^2e^{-\kappa_{n+2}}\norm{f}_{L^\infty(\mathbb{R}^d)}.\end{equation}

We now write $$\left(\tilde{\chi}_{n,x}R_n\right)^{k\ell_n^2}f(x)=\left(\tilde{\chi}_{n,x}S_n+\tilde{\chi}_{n,x}\overline{R}_n\right)^{k\ell_n^2}f(x),$$ and, for nonnegative integers $k_i\geq 0$, \begin{multline*}\left(\tilde{\chi}_{n,x}S_n+\tilde{\chi}_{n,x}\overline{R}_n\right)^{k\ell_n^2}f(x)= \\ \sum_{m=0}^{k\ell_n^2}\sum_{k_0+\ldots+k_m+m=k\ell_n^2}\left(\tilde{\chi}_{n,x}\overline{R}_n\right)^{k_0}\tilde{\chi}_{n,x}S_n\left(\tilde{\chi}_{n,x}\overline{R}_n\right)^{k_1} \ldots \tilde{\chi}_{n,x}S_n\left(\tilde{\chi}_{n,x}\overline{R}_n\right)^{k_m}f(x).\end{multline*}

Since, for each $n\geq 0$, $$\abs{f}_n\leq L_n^\beta\norm{f}_{C^{0,\beta}(\mathbb{R}^d)},$$ and since $x\in B_{4\sqrt{k}\tilde{D}_{n+1}}$ and $\omega\in A_n$, Proposition \ref{prelim_product}, Proposition \ref{prelim_extension}, Control \ref{o_Holder}, Proposition \ref{o_regular}, Proposition \ref{o_contract}, (\ref{o_main_00}) and (\ref{o_main_0}) imply that \begin{multline*}\left|\sum_{m=7}^{k\ell_n^2}\sum_{k_0+\ldots+k_m+m=k\ell_n^2}\left(\tilde{\chi}_{n,x}\overline{R}_n\right)^{k_0}\tilde{\chi}_{n,x}S_n\left(\tilde{\chi}_{n,x}\overline{R}_n\right)^{k_1} \ldots \tilde{\chi}_{n,x}S_n\left(\tilde{\chi}_{n,x}\overline{R}_n\right)^{k_m}f(x)\right|\leq \\ \sum_{m=7}^{k\ell_n^2} {k\ell_n^2 \choose m} 3^mL_n^{\beta-m\delta}\norm{f}_{C^{0,\beta}(\mathbb{R}^d)}.  \end{multline*}  Therefore, for $C>0$ independent of $n$, using (\ref{Holderexponent}) to write $4a+2a^2<5a$, since $1\leq k<\ell_{n+1}^2$, the lefthand side of the above string of inequalities is bounded by \begin{equation}\label{o_main_3} \sum_{m=7}^{k\ell_n^2} \frac{3^m}{m!}L_n^{\beta-m(\delta-5a)}\norm{f}_{C^{0,\beta}(\mathbb{R}^d)}\leq CL_n^{\beta-7(\delta-5a)}\norm{f}_{C^{0,\beta}(\mathbb{R}^d)},\end{equation} where we remark that $\beta-7(\delta-5a)<0$ in view of (\ref{Holderexponent}) and (\ref{delta}).

It remains to consider \begin{equation}\label{o_main_4}\sum_{m=0}^{6}\sum_{k_0+\ldots+k_m+m=k\ell_n^2}\left(\tilde{\chi}_{n,x}\overline{R}_n\right)^{k_0}\tilde{\chi}_{n,x}S_n\left(\tilde{\chi}_{n,x}\overline{R}_n\right)^{k_1} \ldots \tilde{\chi}_{n,x}S_n\left(\tilde{\chi}_{n,x}\overline{R}_n\right)^{k_m}f(x).\end{equation}  We will prove that, up to an error which vanishes as $n$ approaches infinity, the above sum reduces to $$\left(\overline{R}_n\right)^{k\ell_n^2-6}\left(R_n\right)^6f(x).$$  To do so, we consider each summand in $m$ individually.

For the case $m=0$, the single summand is \begin{equation}\label{o_main_5} \left(\tilde{\chi}_{n,x}\overline{R}_n\right)^{k\ell_n^2}f(x).\end{equation}

For the case $m=1$, observe that, since $x\in B_{4\sqrt{k}\tilde{D}_{n+1}}$ and $\omega\in A_n$, Proposition \ref{prelim_product}, Proposition \ref{prelim_extension}, Control \ref{o_Holder}, Proposition \ref{o_regular}, Proposition \ref{o_contract}, (\ref{o_main_00}) and (\ref{o_main_0}) imply that, for $C>0$ independent of $n$, using (\ref{Holderexponent}) to write $4a+2a^2<5a$, \begin{multline}\label{o_main_7}\left|\sum_{k_0+k_1+1=k\ell_n^2}\left(\tilde{\chi}_{n,x}\overline{R}_n\right)^{k_0}\tilde{\chi}_{n,x}S_n\left(\tilde{\chi}_{n,x}\overline{R}_n\right)^{k_1}f(x)-\left(\tilde{\chi}_{n,x}\overline{R}_n\right)^{k\ell_n^2-1}\tilde{\chi}_{n,x}S_nf(x)\right| \\ \leq {k\ell_n^2 \choose 1}3L_n^{-\delta}\norm{f}_{L^\infty(\mathbb{R}^d)}\leq CL_n^{5a-\delta}\norm{f}_{L^\infty(\mathbb{R}^d)},\end{multline} where we observe that $5a-\delta<0$ in view of (\ref{Holderexponent}) and (\ref{delta}).  Furthermore, \begin{equation}\label{o_main_8} \left(\tilde{\chi}_{n,x}\overline{R}_n\right)^{k\ell_n^2-1}\tilde{\chi}_{n,x}S_nf(x)=\left(\tilde{\chi}_{n,x}\overline{R}_n\right)^{k\ell_n^2-1}\tilde{\chi}_{n,x}R_nf(x)-\left(\tilde{\chi}_{n,x}\overline{R}_n\right)^{k\ell_n^2}f(x).\end{equation}  Notice the cancellation between (\ref{o_main_5}) and (\ref{o_main_8}).

In what follows, we use that fact that, for every $f\in L^\infty(\mathbb{R}^d)$, $$\norm{S_nf}_{L^\infty(\mathbb{R}^d)}\leq 2\norm{f}_{L^\infty(\mathbb{R}^d)}.$$  Fix $2\leq m\leq 6$.  In this case, as in the case $m=0$ and $m=1$, Proposition \ref{prelim_product}, Proposition \ref{prelim_extension}, Control \ref{o_Holder}, Proposition \ref{o_regular} and Proposition \ref{o_contract} allow us to reduce the sum to the single term $k_i=0$ for all $1\leq i\leq m$.  Observe that, since $x\in B_{4\sqrt{k}\tilde{D}_{n+1}}$ and $\omega\in A_n$, $$\left|\sum_{k_m\neq 0}\left(\tilde{\chi}_{n,x}\overline{R}_n\right)^{k_0}\tilde{\chi}_{n,x}S_n\ldots\left(\tilde{\chi}_{n,x}\overline{R}_n\right)^{k_m}f(x)\right| \leq {k\ell_n^2 \choose m}3^mL_n^{-m\delta}\norm{f}_{L^\infty(\mathbb{R}^d)}.$$  And, generally, for $1\leq i\leq m$, since $x\in B_{4\sqrt{k}\tilde{D}_{n+1}}$ and $\omega\in A_n$, \begin{multline*}\left|\sum_{k_i\neq 0,\;k_j=0\;\textrm{if}\;j>i}\left(\tilde{\chi}_{n,x}\overline{R}_n\right)^{k_0}\tilde{\chi}_{n,x}S_n\ldots\left(\tilde{\chi}_{n,x}\overline{R}_n\right)^{k_i}\left(\tilde{\chi}_{n,x}S_n\right)^{m-i}f(x)\right| \\ \leq {k\ell_n^2-m+i \choose i}2^{m-i}3^iL_n^{-i\delta}\norm{f}_{L^\infty(\mathbb{R}^d)}.\end{multline*}  Therefore, for $C>0$ independent of $2\leq m\leq 6$ and $n$, using (\ref{Holderexponent}) to write $4a+2a^2<5a$, \begin{multline}\label{o_main_9}\left|\sum_{k_0+\ldots+k_m+m=k\ell_n^2}\left(\tilde{\chi}_{n,x}\overline{R}_n\right)^{k_0}\ldots\left(\tilde{\chi}_{n,x}\overline{R}_n\right)^{k_m}f(x)-\left(\tilde{\chi}_{n,x}\overline{R}_n\right)^{k\ell_n^2-m}\left(\tilde{\chi}_{n,x}S_n\right)^mf(x)\right| \\ \leq \sum_{i=1}^m{k\ell_n^2-m+i \choose i}2^{m-i}3^iL_n^{-i\delta}\norm{f}_{L^\infty(\mathbb{R}^d)}\leq CL_n^{5a-\delta}\norm{f}_{L^\infty(\mathbb{R}^d)},\end{multline} where we observe that $5a-\delta<0$ in view of (\ref{Holderexponent}) and (\ref{delta}).

Furthermore, again using Proposition \ref{prelim_product}, Proposition \ref{prelim_extension}, Control \ref{o_Holder}, Proposition \ref{o_regular}, Proposition \ref{o_contract} and (\ref{o_regular_1}), since $x\in B_{4\sqrt{k}\tilde{D}_{n+1}}$ and $\omega\in A_n$, for each $2\leq m\leq 6$, \begin{multline*}\left|\left(\tilde{\chi}_{n,x}\overline{R}_n\right)^{k\ell_n^2-m}\left(\tilde{\chi}_{n,x}S_n\right)^mf(x)-\left(\tilde{\chi}_{n,x}\overline{R}_n\right)^{k\ell_n^2-m}\left(\tilde{\chi}_{n,x}S_n\right)^{m-1}\tilde{\chi}_{n,x}R_nf(x)\right| \\ =\left|\left(\tilde{\chi}_{n,x}\overline{R}_n\right)^{k\ell_n^2-m}\left(\tilde{\chi}_{n,x}S_n\right)^{m-1}\tilde{\chi}_{n,x}\overline{R}_nf(x)\right|\leq 3^{m-1}L_n^{(m-1)\delta}\norm{f}_{L^\infty(\mathbb{R}^d)}.\end{multline*}  Proceeding inductively, for each $2\leq m\leq 6$, for $C>0$ independent of $m$ and $n$, \begin{multline*}\left|\left(\tilde{\chi}_{n,x}\overline{R}_n\right)^{k\ell_n^2-m}\left(\tilde{\chi}_{n,x}S_n\right)^mf(x)-\left(\tilde{\chi}_{n,x}\overline{R}_n\right)^{k\ell_n^2-m}\tilde{\chi}_{n,x}S_n\left(\tilde{\chi}_{n,x}R_n\right)^{m-1}f(x)\right| \\ \leq CL_n^{-\delta}\norm{f}_{L^\infty(\mathbb{R}^d)},\end{multline*}  where we observe that \begin{multline}\label{o_main_10}\left(\tilde{\chi}_{n,x}\overline{R}_n\right)^{k\ell_n^2}f(x)+\sum_{m=1}^6\left(\tilde{\chi}_{n,x}\overline{R}_n\right)^{k\ell_n^2-m}\tilde{\chi}_{n,x}S_n\left(\tilde{\chi}_{n,x}R_n\right)^{m-1}f(x) \\ =\left(\tilde{\chi}_{n,x}\overline{R}_n\right)^{k\ell_n^2-6}\left(\tilde{\chi}_{n,x}R_n\right)^{6}f(x).\end{multline}  And, since $x\in B_{4\sqrt{k}\tilde{D}_{n+1}}$, $\omega\in A_n$ and $1\leq k<\ell_{n+1}^2$, Control \ref{localization}, Proposition \ref{local11}, Proposition \ref{o_localize}, (\ref{o_main_00}) and (\ref{o_main_0}) imply that there exists $C>0$ and $c>0$ independent of $n$ and such that \begin{equation}\label{o_main_11}\left|\left(\tilde{\chi}_{n,x}\overline{R}_n\right)^{k\ell_n^2-6}\left(\tilde{\chi}_{n,x}R_n\right)^6f(x)-\left(\overline{R}_n\right)^{k\ell_n^2-6}\left(R_n\right)^6f(x)\right|\leq C\ell_{n+1}^2\ell_n^2e^{-c\kappa_{n+2}}\norm{f}_{L^\infty(\mathbb{R}^d)}.\end{equation}

Therefore, in view of (\ref{o_main_0}), (\ref{o_main_2}), (\ref{o_main_3}), (\ref{o_main_5}), (\ref{o_main_7}), (\ref{o_main_9}), (\ref{o_main_10}) and (\ref{o_main_11}), there exits $C>0$ and $c>0$ independent of $n$ such that \begin{multline}\label{o_main_12} \abs{\left(R_{n+1}\right)^kf(x)-\left(\overline{R}_n\right)^{k\ell_n^2-6}\left(R_n\right)^6f(x)}=\abs{\left(R_n\right)^{k\ell_n^2}f(x)-\left(\overline{R}_n\right)^{k\ell_n^2-6}\left(R_n\right)^6f(x)} \\ \leq C\ell_{n+1}^2\ell_n^2e^{-c\kappa_{n+2}}\norm{f}_{L^\infty(\mathbb{R}^d)}+CL_n^{\beta-7(\delta-5a)}\norm{f}_{C^{0,\beta}(\mathbb{R}^d)}+CL_n^{5a-\delta}\norm{f}_{L^\infty(\mathbb{R}^d)}.\end{multline}  In view of (\ref{Holderexponent}), (\ref{L}), (\ref{kappa}) and (\ref{delta}) there exits $C>0$ independent of $n$ such that, for all $n\geq 0$, $$\ell_{n+1}^2\ell_n^2e^{-c\kappa_{n+2}}\leq CL_n^{\beta-7(\delta-5a)}\;\;\textrm{and}\;\;L_n^{5a-\delta}\leq CL_n^{\beta-7(\delta-5a)}.$$  And, since $\norm{f}_{L^\infty(\mathbb{R}^d)}\leq \norm{f}_{C^{0,\beta}(\mathbb{R}^d)}$, we have, using (\ref{o_main_12}), for $C>0$ independent of $n$, \begin{equation}\label{o_main_14}\abs{\left(R_{n+1}\right)^kf(x,\omega)-\left(\overline{R}_n\right)^{k\ell_n^2-6}\left(R_n\right)^6f(x,\omega)}\leq CL_n^{\beta-7(\delta-5a)}\norm{f}_{C^{0,\beta}(\mathbb{R}^d)}.\end{equation}

Finally, since $\omega\in A_n$, using (\ref{o_subset}) and the fact that $6\tilde{D}_n< L_{n+2}^2$, for each $y\in[-L_{n+2}^2,L_{n+2}^2]^d$, we have, using Control \ref{localization} and (\ref{o_localized}), \begin{equation}\label{o_main_15} \abs{\left(R_n\right)^6f(y,\omega)-\left(\tilde{R}_n\right)^6f(y,\omega)}\leq 6e^{-\kappa_{n}}\norm{f}_{L^\infty(\mathbb{R}^d)}.\end{equation}   We write \begin{multline*}\left|\left(\overline{R}_n\right)^{k\ell_n^2-6}\left(R_n\right)^6f(x)-\left(\overline{R}_n\right)^{k\ell_n^2-6}\left(\tilde{R}_n\right)^6f(x)\right| \\ \leq \left|\left(\overline{R}_n\right)^{k\ell_n^2-6}\tilde{\chi}_{n,x}\left(\left(R_n\right)^6-\left(\tilde{R}_n\right)^6\right)f(x)\right|+\left|\left(\overline{R}_n\right)^{k\ell_n^2-6}(1-\tilde{\chi}_{n,x})\left(\left(R_n\right)^6-\left(\tilde{R}_n\right)^6\right)f(x)\right| \end{multline*} and observe that, using Proposition \ref{o_localize}, (\ref{o_main_00}), (\ref{o_main_0}), (\ref{o_main_1}) and (\ref{o_main_15}), since $x\in B_{4\sqrt{k}\tilde{D}_{n+1}}$ and $1\leq k<\ell_{n+1}^2$, for $C_1>0$ and $c_1>0$ independent of $n$, \begin{equation}\label{o_main_16} \left|\left(\overline{R}_n\right)^{k\ell_n^2-6}\left(R_n\right)^6f(x)-\left(\overline{R}_n\right)^{k\ell_n^2-6}\left(\tilde{R}_n\right)^6f(x)\right|\leq C_1e^{-c_1\kappa_n}\norm{f}_{L^\infty(\mathbb{R}^d)}.\end{equation}  Since $\norm{f}_{L^\infty(\mathbb{R}^d)}\leq\norm{f}_{C^{0,\beta}(\mathbb{R}^d)}$ and since (\ref{Holderexponent}), (\ref{L}), (\ref{kappa}) and (\ref{delta}) imply that there exists $C>0$ satisfying, for all $n\geq 0$, $$C_1e^{-c_1\kappa_n}\leq C L_n^{\beta-7(\delta-5a)},$$ we conclude that, in view of (\ref{o_main_14}) and (\ref{o_main_16}), $$\abs{\left(R_{n+1}\right)^kf(x,\omega)-\left(\overline{R}_n\right)^{k\ell_n^2-6}\left(\tilde{R}_n\right)^6f(x,\omega)}\leq CL_n^{\beta-7(\delta-5a)}\norm{f}_{C^{0,\beta}(\mathbb{R}^d)}.$$  Since $n\geq 0$, $1\leq k<\ell_{n+1}^2$, $\omega\in A_n$, $x\in B_{4\sqrt{k}\tilde{D}_{n+1}}$ and $f\in C^{0,\beta}(\mathbb{R}^d)$ were arbitrary, this completes the proof.  \end{proof}

We now prepared to provide the initial characterization of the invariant measure $\pi:\mathcal{F}\rightarrow\mathbb{R}$.  In view of Proposition \ref{o_main}, for each $n\geq 0$ and $f\in L^\infty(\mathbb{R}^d\times\Omega)$, define \begin{equation}\label{o_pin} \pi_n(f)=\mathbb{E}\left(\left(\tilde{R}_n\right)^6R_1f(0,\omega)\right).\end{equation}  The following two propositions prove that, for each $f\in L^\infty(\mathbb{R}^d\times\Omega)$ satisfying (\ref{o_stationary}), the sequence $\left\{\pi_n(f)\right\}_{n=0}^\infty$ is Cauchy.  Notice in particular that the rate of convergence depends only upon the $L^\infty$ norm of the initial condition.

\begin{prop}\label{o_cauchy}  Assume (\ref{steady}) and (\ref{constants}).  For each $n\geq 0$ and $f\in L^\infty(\mathbb{R}^d\times\Omega)$ satisfying (\ref{o_stationary}), for $C>0$ independent of $n$, $$\abs{\pi_{n+1}(f)-\pi_n(f)}\leq CL_n^{\beta-7(\delta-5a)}\norm{f}_{L^\infty(\mathbb{R}^d\times\Omega)}.$$ \end{prop}

\begin{proof}  Fix $n\geq 0$ and $f\in L^\infty(\mathbb{R}^d\times\Omega)$ satisfying (\ref{o_stationary}).  Since $$\norm{R_1f}_{L^\infty(\mathbb{R}^d\times\Omega)}\leq\norm{f}_{L^\infty(\mathbb{R}^d\times\Omega)},$$ we have, for each $\omega\in A_n$, using Control \ref{localization} and (\ref{o_localized}), \begin{equation}\label{o_cauchy_1} \abs{\left(R_{n+1}\right)^6R_1f(0,\omega)-\left(\tilde{R}_{n+1}\right)^6R_1f(0,\omega)}\leq 6e^{-\kappa_{n+1}} \norm{f}_{L^\infty(\mathbb{R}^d\times\Omega)},\end{equation} and, using Proposition \ref{o_weak} and Proposition \ref{o_main} for $k=6$, for $C>0$ independent of $n$ and $f$, \begin{multline}\label{o_cauchy_2}\abs{\left(R_{n+1}\right)^6R_1f(0,\omega)-\left(\overline{R}_n\right)^{6\ell_n^2-6}\left(\tilde{R}_n\right)^6R_1f(0,\omega)}\leq CL_n^{\beta-7(\delta-5a)}\norm{R_1f(x,\omega)}_{C^{0,\beta}(\mathbb{R}^d)} \\ \leq CL_n^{\beta-7(\delta-5a)}\norm{f}_{L^\infty(\mathbb{R}^d\times\Omega)}.\end{multline}  Since (\ref{Holderexponent}), (\ref{L}), (\ref{kappa}) and (\ref{delta}) imply that there exists $C>0$ satisfying, for each $n\geq 0$, $$e^{-\kappa_{n+1}}\leq CL_n^{\beta-7(\delta-5a)},$$ we have, for each $\omega\in A_n$, in view of (\ref{o_cauchy_1}) and (\ref{o_cauchy_2}), for $C>0$ independent of $n$, \begin{equation}\label{o_cauchy_3} \abs{\left(\tilde{R}_{n+1}\right)^6R_1f(0,\omega)-\left(\overline{R}_n\right)^{6\ell_n^2-6}\left(\tilde{R}_n\right)^6R_1f(0,\omega)}\leq CL_n^{\beta-7(\delta-5a)}\norm{f}_{L^\infty(\mathbb{R}^d\times\Omega)}.\end{equation}

Therefore, since (\ref{o_stationary}), Proposition \ref{o_sol} and Proposition \ref{o_localsol} imply that, for each $x\in\mathbb{R}^d$, $$\mathbb{E}\left(\left(\overline{R}_n\right)^{6\ell_n^2-6}\left(\tilde{R}_n\right)^6R_1f(x,\omega)\right)=\pi_n(f),$$ and since (\ref{Holderexponent}) and (\ref{delta}) imply that, for each $n\geq 0$, $$L_{n}^{(2(1+a)^2-1)d-M_0}\leq L_n^{\beta-7(\delta-5a)},$$ Proposition \ref{o_probability} and (\ref{o_cauchy_3}) imply that, for $C>0$ independent of $n$, \begin{multline*}\abs{\pi_{n+1}(f)-\pi_n(f)}\leq CL_n^{\beta-7(\delta-5a)}\norm{f}_{L^\infty(\mathbb{R}^d\times\Omega)}+2\norm{f}_{L^\infty(\mathbb{R}^d\times\Omega)}\mathbb{P}\left(\Omega\setminus A_n\right) \\ \leq CL_n^{\beta-7(\delta-5a)}\norm{f}_{L^\infty(\mathbb{R}^d\times\Omega)},\end{multline*} which, since $n\geq 0$ and $f\in L^\infty(\mathbb{R}^d\times\Omega)$ were arbitrary, completes the argument.  \end{proof}

\begin{prop}\label{o_identify}  Assume (\ref{steady}) and (\ref{constants}).  For each $f\in L^\infty(\mathbb{R}^d\times\Omega)$ satisfying (\ref{o_stationary}) there exists a unique $\overline{\pi}(f)\in\mathbb{R}$ satisfying $$\overline{\pi}(f)=\lim_{n\rightarrow\infty}\pi_n(f).$$  Furthermore, for each $n\geq 0$, for $C>0$ independent of $n$ and $f$, $$\abs{\pi_n(f)-\overline{\pi}(f)}\leq C\norm{f}_{L^\infty(\mathbb{R}^d\times\Omega)}\sum_{m=n}^\infty L_m^{\beta-7(\delta-5a)}.$$\end{prop}

\begin{proof}  In view of (\ref{Holderexponent}), (\ref{L}) and (\ref{delta}), since $\beta-7(\delta-5a)<0$, the ratio test implies that $$\sum_{m=0}^\infty L_m^{\beta-7(\delta-5a)}<\infty.$$  Therefore, for each $f\in L^\infty(\mathbb{R}^d\times\Omega)$ satisfying (\ref{o_stationary}), Proposition \ref{o_cauchy} implies that the sequence $\left\{\pi_n(f)\right\}_{n=1}^\infty$ is Cauchy.  Therefore, for each $f\in L^\infty(\mathbb{R}^d\times\Omega)$ satisfying (\ref{o_stationary}), there exists a unique $\overline{\pi}(f)\in\mathbb{R}$ such that \begin{equation}\label{o_identify_1} \lim_{n\rightarrow\infty}\pi_n(f)=\overline{\pi}(f).\end{equation}  Furthermore, for each $f\in L^\infty(\mathbb{R}^d\times\Omega)$ satisfying (\ref{o_stationary}), the triangle inequality, Proposition \ref{o_cauchy} and (\ref{o_identify_1}) imply that, for each $n\geq 0$, for $C>0$ independent of $n$ and $f$, $$\abs{\pi_n(f)-\overline{\pi}(f)}\leq \sum_{m=n}^\infty\abs{\pi_{m+1}(f)-\pi_m(f)}\leq C\norm{f}_{L^\infty(\mathbb{R}^d\times\Omega)}\sum_{m=n}^\infty L_m^{\beta-7(\delta-5a)},$$ which completes the argument.  \end{proof}

We now define what we show in the next section to be the unique invariant measure.  For every $E\in\mathcal{F}$, write ${\bf{1}}_E:\Omega\rightarrow \mathbb{R}$ for the indicator function of $E\subset\Omega$, and define $f_E:\mathbb{R}^d\times\Omega\rightarrow\mathbb{R}$ using the translation group $\left\{\tau_x\right\}_{x\in\mathbb{R}^d}$, see (\ref{transgroup}), \begin{equation}\label{o_translate} f_E(x,\omega)={\bf{1}}_E(\tau_x\omega).\end{equation}  We define $\pi:\mathcal{F}\rightarrow \mathbb{R}$, for each $E\in\mathcal{F}$, by the rule \begin{equation}\label{o_p} \pi(E)=\overline{\pi}(f_E),\end{equation} and prove now that $\pi$ defines a probability measure on $(\Omega,\mathcal{F})$ which is absolutely continuous with respect to $\mathbb{P}$.

\begin{prop}\label{o_pmeasure}  Assume (\ref{steady}) and (\ref{constants}).  The function $\pi:\mathcal{F}\rightarrow\mathbb{R}$ defined in (\ref{o_p}) defines a probability measure on $(\Omega,\mathcal{F})$ which is absolutely continuous with respect to $\mathbb{P}$.\end{prop}

\begin{proof}  For each $E\in\mathcal{F}$, since $0\leq f_E\leq 1$ on $\mathbb{R}^d\times\Omega$, the comparison principle implies that, for each $n\geq 0$, $$0\leq \pi_n(f_E)\leq 1, $$ and, therefore, for each $E\in\mathcal{F}$, \begin{equation}\label{o_pmeasure_1} 0\leq \pi(E)=\overline{\pi}(f_E)\leq 1.\end{equation}  Furthermore, since $f_\Omega$ is identically one and, since $f_\emptyset$ is identically zero, we have, for each $n\geq 0$, $$\pi_n(f_\Omega)=1\;\;\textrm{and}\;\;\pi_n(f_\emptyset)=0.$$  Therefore, \begin{equation}\label{o_pmeasure_2}\pi(\Omega)=1\;\;\textrm{and}\;\;\pi(\emptyset)=0.\end{equation}  It remains to prove that $\pi$ is countably additive and absolutely continuous.

Let $\left\{A_i\right\}_{i=1}^\infty\subset\mathcal{F}$ be a countable collection of disjoint subsets.  Since, for each $n\geq 0$ and $1\leq m\leq \infty$, $$\pi_n(f_{\bigcup_{i=1}^m A_i})=\int_{\Omega}\int_{\C([0,\infty);\mathbb{R}^d)}R_1f_{\bigcup_{i=1}^m A_i}(X_{6L_n^2\wedge T_n},\omega)\;dP_{0,\omega}d\mathbb{P},$$ for the stopping time $$T_n=\inf\left\{\;s\geq 0\;|\; X_s\notin B_{6\tilde{D}_n}\;\right\},$$ the dominated convergence theorem implies that, for each $n\geq 0$, there exists $k_n\geq n$ such that $$ \abs{\pi_n(f_{\bigcup_{i=1}^\infty A_i})-\pi_n(f_{\bigcup_{i=1}^{k_n}A_i})}\leq \frac{1}{n}.$$  Therefore, in view of Proposition \ref{o_identify}, since each initial condition has unit $L^\infty$ norm, for $C>0$ independent of $n$, \begin{equation}\label{o_pmeasure_4} \left|\pi\left(\bigcup_{i=1}^\infty A_i\right)-\pi\left(\bigcup_{i=1}^{k_n}A_i\right)\right|\leq \frac{1}{n}+C\sum_{m=n}^\infty L_m^{\beta-7(\delta-5a)}.\end{equation}  Furthermore, since the $\left\{A_i\right\}_{i=1}^\infty$ are disjoint, for each $n\geq 0$, \begin{equation}\label{o_pmeasure_5} \pi\left(\bigcup_{i=1}^{k_n}A_i\right)=\sum_{i=1}^{k_n}\pi(A_i).\end{equation}  Therefore, in view of (\ref{o_pmeasure_4}) and (\ref{o_pmeasure_5}), since we choose $k_n\geq n$, \begin{multline*}\lim_{n\rightarrow\infty}\left|\pi\left(\bigcup_{i=1}^\infty A_i\right)-\pi\left(\bigcup_{i=1}^{k_n}A_i\right)\right|=\lim_{n\rightarrow\infty}\left|\pi\left(\bigcup_{i=1}^\infty A_i\right)-\sum_{i=1}^{k_n}\pi\left(A_i\right)\right| \\ =\left|\pi\left(\bigcup_{i=1}^\infty A_i\right)-\sum_{i=1}^\infty\pi(A_i)\right|=0,\end{multline*} which, since the family $\left\{A_i\right\}_{i=1}^\infty$ was arbitrary, completes the proof of countable additivity.

We now prove the absolute continuity.  We first show that whenever $E\in\mathcal{F}$ satisfies $\mathbb{P}(E)=0$ we have $R_1f_E(x,\omega)=0$ on $\mathbb{R}^d$ for almost every $\omega\in\Omega$.  To do so, we recall that there exists a density $p(x,1,y,\omega)$ satisfying for each $x\in\mathbb{R}^d$, $\omega\in\Omega$ and $E\in\mathcal{F}$, $$R_1f_E(x,\omega)=\int_{\mathbb{R}^d}p(x,1,y,\omega)f_E(y,\omega)\;dy.$$  Furthermore, for each $x\in\mathbb{R}^d$ and $\omega\in\Omega$, the probability measure defined by $p(x,1,y,\omega)\;dy$ on $\mathbb{R}^d$ is equivalent to Lebesgue measure.  See, for instance, \cite{Fr}.

Fix $E\in\mathcal{F}$ satisfying $\mathbb{P}(E)=0$.  Then, for each $x\in\mathbb{R}^d$, using (\ref{transgroup}) and $\mathbb{P}(E)=0$, by Fubini's theorem $$\mathbb{E}\left(R_1f_E(x,\omega)\right)=\int_{\Omega}\int_{\mathbb{R}^d}p(x,1,y,\omega){\bf{1}}_E(\tau_y\omega)\;dyd\mathbb{P}=\int_{\mathbb{R}^d}\int_{\Omega}p(x,1,y,\omega){\bf{1}}_E(\tau_y\omega)\;d\mathbb{P}dy=0,$$ since ${\bf{1}}_E(\tau_y\omega)=0$ almost everywhere in $\Omega$ for every $y\in\mathbb{R}^d$.  Therefore, Fubini's theorem implies that, for every $x\in\mathbb{R}^d$, there exists a subset $A_x\subset\Omega$ of full probability such that, for every $\omega\in A_x$, $$R_1f_E(x,\omega)=0.$$  Define the subset of full probability $$A=\bigcap_{x\in\mathbb{Q}^d} A_x,$$ and observe that, for each $\omega\in A$ and $x\in\mathbb{Q}^d$, $$R_1f_E(x,\omega)=0.$$  Since Proposition \ref{o_weak} implies that, for every $\omega\in\Omega$, we have $R_1f_E(x,\omega)\in C^{0,\beta}(\mathbb{R}^d)$, we conclude that, for every $x\in\mathbb{R}^d$ and $\omega\in A$, $$R_1f_E(x,\omega)=0,$$ and, therefore, for every $\omega\in A$ and $n\geq 0$, $$\left(\tilde{R}_n\right)^6R_1f_E(0,\omega)=0.$$  Since $\mathbb{P}(A)=1$, this implies that, for each $n\geq 0$, $\pi_n(f_E)=0$ and, therefore, that $\pi(E)=0$.  Since $E\in\mathcal{F}$ satisfying $\mathbb{P}(E)=0$ was arbitrary, this completes the argument.  \end{proof}

In the final proposition of this section, we prove that for each $f\in L^\infty(\mathbb{R}^d\times\Omega)$ satisfying (\ref{o_stationary}), the constant $\overline{\pi}(f)$ characterizes the integral of $f(0,\omega)$ with respect to $\pi$.  This is essentially an immediate consequence of the definition of $\pi$ and the fact that the kernels $R_t$ preserve the $L^\infty$ norm of initial data.

\begin{prop}\label{o_integral}  Assume (\ref{steady}) and (\ref{constants}).  For every $f\in L^\infty(\mathbb{R}^d\times\Omega)$ satisfying (\ref{o_stationary}), $$\overline{\pi}(f)=\int_{\Omega}f(0,\omega)\;d\pi.$$\end{prop}

\begin{proof}  By definition, see (\ref{o_p}), for every subset $E\in\mathcal{F}$, for $f_E$ defined in (\ref{o_translate}), \begin{equation}\label{o_integral_1}\overline{\pi}(f_E)=\pi(E)=\int_{\Omega}f(0,\omega)\;d\pi=\int_{\Omega}{\bf{1}}_E(\omega)\;d\pi.\end{equation}  And, since for every $t\geq 0$, $n\geq 0$ $\omega\in\Omega$ and $f\in L^\infty(\mathbb{R}^d)$, $$\norm{R_tf(x,\omega)}_{L^\infty(\mathbb{R}^d)}\leq \norm{f}_{L^\infty(\mathbb{R}^d)}\;\;\textrm{and}\;\;\norm{\tilde{R}_nf(x,\omega)}_{L^\infty(\mathbb{R}^d)}\leq \norm{f}_{L^\infty(\mathbb{R}^d)},$$  we have, for each $n\geq 0$ and $f,g\in L^\infty(\mathbb{R}^d\times\Omega)$ satisfying (\ref{o_stationary}), $$\abs{\pi_n(f-g)}=\abs{\pi_n(f)-\pi_n(g)}\leq \norm{f-g}_{L^\infty(\mathbb{R}^d\times\Omega)}.$$  Therefore, for every pair $f,g\in L^\infty(\mathbb{R}^d\times\Omega)$ satisfying (\ref{o_stationary}), \begin{equation}\label{o_integral_2}\abs{\overline{\pi}(f)-\overline{\pi}(g)}\leq \norm{f-g}_{L^\infty(\mathbb{R}^d\times\Omega)}.\end{equation}  The claim now follows from (\ref{o_integral_1}), (\ref{o_integral_2}) and the definition fo the Lebesgue integral.  \end{proof}

\section{The Proof of Invariance and Uniqueness}

In this section, we prove that the measure $\pi$ defined in (\ref{o_p}) is the unique invariant measure which is absolutely continuous with respect to $\mathbb{P}$.  Furthermore, if the transformation group $\left\{\tau_x\right\}_{x\in\mathbb{R}^d}$ is ergodic, then $\pi$ is mutually absolutely continuous with respect to $\mathbb{P}$ and defines an ergodic probability measure for the canonical Markov process on $\Omega$ defining (\ref{i_initial}).  We observe that, for each $t\geq 0$, $\omega\in\Omega$ and $E\in\mathcal{F}$, for $P_t(\omega,E)$ defined in (\ref{i_initial}), $$R_tf_E(0,\omega)=P_{0,\omega}(\tau_{X_t}\omega\in E)=P_t(\omega, E).$$  In order to prove invariance, therefore, it suffices to prove that, for each $t\geq 0$ and $E\in\mathcal{F}$, $$\pi(E)=\int_{\Omega}R_tf_E(0,\omega)\;d\pi.$$  See Proposition  \ref{u_invariant}.

In order to exploit the finite range dependence, see (\ref{finitedep}), we define, for each $R>0$, $t\geq 1$ and $\omega\in\Omega$, the localized kernels \begin{equation}\label{u_local} \tilde{R}_{t,R}f(x,\omega)=\tilde{u}_R(x,t,\omega),\end{equation} for $\tilde{u}:\overline{B}_R(x)\times[0,\infty)\times\Omega\rightarrow R$ satisfying \begin{equation}\label{u_local_1} \left\{\begin{array}{ll}\tilde{u}_{R,t}-\frac{1}{2}\tr(A(y,\omega)D^2\tilde{u}_R)+b(y,\omega)\cdot D\tilde{u}_R=0 & \textrm{on}\;\; B_R(x)\times(0,\infty), \\ \tilde{u}_R=f & \textrm{on}\;\;B_R(x)\times\left\{0\right\}\cup \partial B_R(x)\times[0,\infty).\end{array}\right.\end{equation}  The following proposition controls the error we make due to this localization.  And, in contrast to Control \ref{localization}, we obtain this control globally for $x\in\mathbb{R}^d$ and $\omega\in\Omega$.  Notice, however, that this control is only effective at length scales which are significantly larger than those appearing in Control \ref{localization}.

\begin{prop}\label{u_martingale}  Assume (\ref{steady}).  For each $x\in\mathbb{R}^d$, $t\geq 1$, $\omega\in\Omega$ and $R>0$,  for $C>0$ independent of $x$, $t$, $\omega$ and $R$, for every $f\in L^\infty(\mathbb{R}^d)$, $$\abs{R_tf(x,\omega)-\tilde{R}_tf(x,\omega)}\leq \norm{f}_{L^\infty(\mathbb{R}^d)}e^{-\frac{(R-Ct)_+^2}{Ct}}$$ \end{prop}

\begin{proof}  Fix $R>0$, $t\geq 1$, $x\in\mathbb{R}^d$ and $\omega\in\Omega$.  Introduce the stopping time $T_R:\C([0,\infty);\mathbb{R}^d)\rightarrow\mathbb{R}$ defined by $$T_R=\inf\left\{\;s\geq 0\;|\;X_s\notin B_R(x)\;\right\}.$$  Then, for each $f\in L^\infty(\mathbb{R}^d)$, for $X^*_t$ defined in (\ref{star}), \begin{equation}\label{u_martingale_1}\abs{R_tf(x,\omega)-\tilde{R}_{t,R}f(x,\omega)}\leq \norm{f}_{L^\infty(\mathbb{R}^d)}P_{x,\omega}\left(X^*_t\geq R\right).\end{equation}

We recall that, almost surely with respect to $P_{x,\omega}$, for  $B_s$ a Brownian motion on $\mathbb{R}^d$ under $P_{x,\omega}$ with respect to the canonical right-continuous filtration on $\C([0,\infty);\mathbb{R}^d)$, paths $X_s\in\C([0,\infty);\mathbb{R}^d)$ satisfy the stochastic differential equation $$\left\{\begin{array}{l} dX_s=-b(X_s,\omega)dt+\sigma(X_s,\omega)dB_s, \\ X_0=x.\end{array}\right.$$  Therefore, using the exponential inequality for Martingales, see Revuz, Yor \cite{RY}, and (\ref{bounded}) and (\ref{Lipschitz}), for every $\tilde{R}\geq 0$, for $C>0$ independent of $\tilde{R}$, $t$, $x$ and $\omega$, \begin{equation}\label{u_martingale_2}P_{x,\omega}\left(X^*_t\geq \tilde{R}+Ct\right)\leq e^{-\frac{\tilde{R}^2}{Ct}}.\end{equation}

Therefore, by choosing $\tilde{R}=(R-CT)_+$ in (\ref{u_martingale_2}), we conclude in view of (\ref{u_martingale_1}) that, for $C>0$ independent of $x$, $t$, $\omega$ and $R$, $$\abs{R_tf(x,\omega)-\tilde{R}_{t,R}f(x,\omega)}\leq \norm{f}_{L^\infty(\mathbb{R}^d)}e^{-\frac{(R-Ct)_+^2}{Ct}},$$ which, since $x$, $t$, $\omega$ and $R$ were arbitrary, completes the argument.  \end{proof}

We define, for each subset $A\subset\mathbb{R}^d$, the sub sigma algebra of $\mathcal{F}$ \begin{equation}\label{u_sigma} \sigma_A=\sigma\left(A(x,\omega), b(x,\omega)\;|\;x\in A\right).\end{equation}  The following proposition uses stationary, see (\ref{stationary}), to describe the interaction between the transformation group $\left\{\tau_x\right\}_{x\in\mathbb{R}^d}$ and the sigma algebras $\sigma_A$.

\begin{prop}\label{u_sigma_translate}  Assume (\ref{steady}).  For every subset $A\subset\mathbb{R}^d$ and $y\in\mathbb{R}^d$, $$\tau_y\left(\sigma_A\right)=\left\{\;\tau_x(B)\;|\;B\in \sigma_A\;\right\}=\sigma_{A-y}.$$\end{prop}

\begin{proof}  Fix $A\subset\mathbb{R}^d$.  We identify $\mathcal{S}(d)$ with $\mathbb{R}^{(d+1)d/2}$ and write $\mathcal{B}_d$ and $\mathcal{B}_{(d+1)d/2}$ for the Borel sigma algebras on $\mathbb{R}^d$ and $\mathbb{R}^{(d+1)d/2}$ respectively.  Observe that $\sigma_A$ is generated by sets of the form, for fixed $x\in A$, $B_d\in \mathcal{B}_d$ and $B_{(d+1)d/2}\in\mathcal{B}_{(d+1)d/2}$, \begin{equation}\label{u_sigma_translate_1} b(x,\omega)^{-1}(B_d)=\left\{\;\omega\in\Omega\;|\;b(x,\omega)\in B_d\;\right\},\end{equation} and \begin{equation}\label{u_sigma_translate_2}A(x,\omega)^{-1}(B_{(d+1)d/2})=\left\{\;\omega\in\Omega\;|\;A(x,\omega)\in B_{(d+1)d/2}\;\right\}.\end{equation}  Furthermore, since the group $\left\{\tau_y\right\}_{y\in\mathbb{R}^d}$ is composed of invertible, measure-preserving transformations, for every fixed $y\in\mathbb{R}^d$, $$\tau_y\left(\sigma_A\right)=\left\{\;\tau_yB\;|\;B\in\sigma_A\;\right\},$$ is a sigma algebra generated by sets of the form, for fixed $x\in A$, $B_d\in \mathcal{B}_d$ and $B_{(d+1)d/2}\in\mathcal{B}_{(d+1)d/2}$, \begin{equation}\label{u_sigma_translate_3}\tau_y\left(b(x,\omega)^{-1}(B_d)\right)\;\;\textrm{and}\;\;\tau_y\left(A(x,\omega)^{-1}(B_{(d+1)d/2})\right).\end{equation}  And, in view of (\ref{stationary}), for each $y\in\mathbb{R}^d$, $x\in A$, $B_d\in \mathcal{B}_d$ and $B_{(d+1)d/2}\in\mathcal{B}_{(d+1)d/2}$, \begin{equation}\label{u_sigma_translate_4}\tau_y\left(b(x,\omega)^{-1}(B_d)\right)=b(x,\tau_{-y}\omega)^{-1}(B_d)=b(x-y,\omega)^{-1}(B_d),\end{equation} and \begin{equation}\label{u_sigma_translate_5}\tau_y\left(A(x,\omega)^{-1}(B_{(d+1)d/2})\right)=A(x,\tau_{-y}\omega)^{-1}(B_{(d+1)d/2})=A(x-y,\omega)^{-1}(B_{(d+1)d/2}).\end{equation}  We therefore conclude, using (\ref{u_sigma_translate_1}), (\ref{u_sigma_translate_2}), (\ref{u_sigma_translate_3}), (\ref{u_sigma_translate_4}) and (\ref{u_sigma_translate_5}), for each $y\in\mathbb{R}^d$, \begin{equation}\label{i_finite_6} \tau_y\left(\sigma_A\right)=\tau_y\sigma\left( A(x-y,\omega), b(x-y,\omega)\;|\;x\in A \right)=\sigma_{A-y}.\end{equation}  Since $A\subset\mathbb{R}^d$ was arbitrary, this completes the argument.  \end{proof}

In what follows, we will use the fact that $$\mathcal{F}=\sigma\left(A(x,\omega), b(x,\omega)\;|\;x\in\mathbb{R}^d\right)=\sigma\left(\bigcup_{R>0}\sigma_{B_R}\right).$$  This will allow us to obtain our general statement after considering measurable subsets $E\subset\Omega$ in the algebra of subsets $\cup_{R>0}\sigma_{B_R}$, where it is shown in the next proposition that for these subsets we can effectively apply the finite range dependence, see (\ref{finitedep}).

\begin{prop}\label{u_localize}  Assume (\ref{steady}) and (\ref{constants}).  Suppose that, for $R_1>0$, $E\in\mathcal{F}$ satisfies $E\in \sigma_{B_{R_1}}$.  For each $x\in\mathbb{R}^d$, $t\geq 1$ and $R_2>0$, see (\ref{o_translate}) and (\ref{u_local}), $$\sigma\left(\tilde{R}_{t,R_2}f_E(x,\omega)\right)\subset\sigma_{B_{R_2+R_1}(x)}.$$  \end{prop}

\begin{proof}  Fix $E\in\mathcal{F}$ and $R_1>0$ satisfying $E\in\sigma_{B_{R_1}}$, $x\in\mathbb{R}^d$, $t\geq 1$ and $R_2>0$.  In view of the definition (\ref{u_local}), \begin{equation}\label{u_localize_1}\sigma\left(\tilde{R}_{t,R_2}f_E(x,\omega)\right)\subset\sigma\left(A(y,\omega), b(y,\omega), f_E(y,\omega)\;|\;y\in B_{R_2}(x)\right)\end{equation}  And, in view of (\ref{o_translate}) and Proposition \ref{u_sigma_translate}, since $E\in \sigma_{B_{R_1}}$, for every $y\in \mathbb{R}^d$, \begin{equation}\label{u_localize_2}\sigma(f_E(y,\omega))=\left\{\;\tau_{-y}(E), \Omega\setminus\tau_{-y}\left(E\right)\;\right\}\subset \tau_{-y}(\sigma_{B_{R_1}})=\sigma_{B_{R_1}+y}.\end{equation}  Therefore, in view of (\ref{u_localize_1}) and (\ref{u_localize_2}), $$\sigma\left(\tilde{R}_{t,R_2}f_E(x,\omega)\right)\subset\sigma\left(\bigcup_{y\in B_{R_2}(x)}\sigma_{B_{R_1}+y}\right)=\sigma_{B_{R_2}(x)+B_{R_1}}=\sigma_{B_{R_2+R_1}(x)},$$ which, since $R_1$, $E\in \sigma_{B_{R_1}}$, $x$, $t$, and $R_2$ were arbitrary, completes the argument.  \end{proof}

Recall that in Proposition \ref{o_main}, with high probability, we obtained an effective comparison between the kernels $$R_{n+1}\;\;\textrm{and}\;\;\overline{R}_n^{\ell_n^2-6}R_n^6,$$ where, in view of Theorem \ref{effectivediffusivity}, the expectation is that the presence of the heat kernel will result  significant averaging for appropriate initial data.  The following proposition quantifies the effect of this averaging.

\begin{prop}\label{u_var}  Assume (\ref{steady}) and (\ref{constants}).  Suppose that, for $R_1>0$, $E\in\mathcal{F}$ satisfies $E\in\sigma_{B_{R_1}}$.  For each $n\geq 0$, $1\leq k<\ell_{n}^2$ and $t\geq 0$ there exists $C=C(t,R_1)>0$ independent of $E$, $n$ and $k$, and there exists $\zeta>0$ independent of $R_1$, $E$, $n$, $k$ and $t$, such that $$\mathbb{P}\left(\sup_{x\in B_{4\sqrt{k}\tilde{D}_{n+1}}}\left|\left(\overline{R}_n\right)^{k\ell_n^2-6}\left(\tilde{R}_n\right)^6R_{1+t}f_E(x,\omega)-\pi_n(R_tf_E)\right|>\frac{C}{\tilde{\kappa}_n}\right)\leq Ck^{-(1+\zeta)}L_n^{-\zeta}.$$\end{prop}

\begin{proof}  Fix $E\in\mathcal{F}$ and $R_1>0$ satisfying $E\in\sigma_{B_{R_1}}$, $t\geq 0$, $n\geq 0$ and $1\leq k<\ell_n^2$.  We define \begin{equation}\label{u_var_0} R_2=6\tilde{D}_n,\end{equation} and observe that, in view of Proposition \ref{u_martingale}, for every $x\in\mathbb{R}^d$ and $\omega\in\Omega$, for $C_1>0$ independent of $n$, $k$, $x$, $t$ and $\omega$, \begin{equation}\label{u_var_00} \left|R_{1+t}f_E(x\,\omega)-\tilde{R}_{1+t,R_2}f_E(x,\omega)\right|\leq \norm{f_E}_{L^\infty(\mathbb{R}^d\times\Omega)}e^{-\frac{(6\tilde{D}_n-C_1(1+t))_+^2}{C_1(1+t)}}\leq e^{-\frac{(6\tilde{D}_n-C_1(1+t))_+^2}{C_1(1+t)}}.\end{equation}  In order to obtain better localization properties, we therefore consider the quantity \begin{equation}\label{u_var_1}\left(\overline{R}_n\right)^{k\ell_n^2-6}\left(\tilde{R}_n\right)^6\tilde{R}_{1+t,R_2}f_E(0,\omega)=\int_{\mathbb{R}^d}p_{n,k}(y)\left(\tilde{R}_n\right)^6\tilde{R}_{t+1,R_2}f_E(y,\omega)\;dy,\end{equation} where, for each $y\in\mathbb{R}^d$, $$p_{n,k}(y)=(4\pi \alpha_n(kL_{n+1}^2-6L_n^2))^{-d/2}e^{-\abs{y}^2/4\alpha_n (kL_{n+1}^2-6L_n^2)}.$$  Here, we observe that Theorem \ref{effectivediffusivity} implies, for $C>0$ independent of $n$, \begin{equation}\label{u_var_2}\norm{p_{n,k}(y)}_{L^\infty(\mathbb{R}^d)}\leq Ck^{-d/2}L_{n+1}^{-d}.\end{equation}

We now define, for each $n\geq 0$, \begin{equation}\label{u_var_000} \tilde{\pi}_n(R_tf_E)=\mathbb{E}\left(\left(\overline{R}_n\right)^{k\ell_n^2-6}\left(\tilde{R}_n\right)^6\tilde{R}_{1+t,R_2}f_E(0,\omega)\right)=\mathbb{E}\left(\left(\tilde{R}_n\right)^6\tilde{R}_{1+t,R_2}f_E(0,\omega)\right),\end{equation} and observe that, in view of (\ref{u_var_00}), for each $n\geq 0$, \begin{equation}\label{u_var_0000} \abs{\pi_n(R_tf_E)-\tilde{\pi}_n(R_tf_E)}\leq e^{-\frac{(6\tilde{D}_n-C(1+t))_+^2}{C(1+t)}}.\end{equation}  Furthermore, using the stationary (\ref{stationary}), since $f_E$ satisfies (\ref{o_stationary}), we have, for each $x\in\mathbb{R}^d$, \begin{equation}\label{u_var_3} \tilde{\pi}_n(R_tf_E)=\mathbb{E}\left(\left(\overline{R}_n\right)^{k\ell_n^2-6}\left(\tilde{R}_n\right)^6\tilde{R}_{1+t,R_2}f_E(x,\omega)\right)=\mathbb{E}\left(\left(\tilde{R}_n\right)^6\tilde{R}_{1+t, R_2}f_E(x,\omega)\right).\end{equation}

Observe that, definition (\ref{o_localized}) and the choice of $R_2=6\tilde{D}_n$ in (\ref{u_var_0}) imply that, for each $x\in\mathbb{R}^d$ and $\omega\in\Omega$, $$\left(\tilde{R}_n\right)^6\tilde{R}_{t+1,R_2}f_E(x,\omega)=\tilde{R}_{6\tilde{D}_n+1+t,R_2}f_E(x,\omega),$$ and, using Proposition \ref{o_localize}, for each $x\in\mathbb{R}^d$, \begin{equation}\label{u_var_00000} \sigma\left(\left(\tilde{R}_n\right)^6\tilde{R}_{t+1,R_2}f_E(x,\omega)\right)\subset\sigma_{B_{R_2+R_1}(x)}=\sigma_{B_{6\tilde{D}_n+R_1}(x)}.\end{equation}  Therefore, for $R>0$ as in (\ref{finitedep}), whenever $x,y\in\mathbb{R}^d$ satisfy $\abs{x-y}\geq 12\tilde{D}_n+2R_1+R,$ the random variables \begin{equation}\label{u_var_000000} \left(\tilde{R}_n\right)^6\tilde{R}_{t+1,R_2}f_E(x,\omega)\;\;\textrm{and}\;\;\left(\tilde{R}_n\right)^6\tilde{R}_{t+1,R_2}f_E(y,\omega)\;\;\textrm{are independent.}\end{equation}

We now write $\tilde{\pi}_n=\tilde{\pi}_n(R_tf)$ and compute the variance \begin{multline}\label{u_var_4} \mathbb{E}\left(\left(\left(\overline{R}_n\right)^{k\ell_n^2-6}\left(\tilde{R}_n\right)^6\tilde{R}_{t+1,R_2}f_E(0,\omega)-\tilde{\pi}_n\right)^2\right)= \\ \mathbb{E}\left(\int_{\mathbb{R}^d}\int_{\mathbb{R}^d}p_{n,k}(y)p_{n,k}(z)\left(\left(\tilde{R}_n\right)^6\tilde{R}_{t+1,R_2}f_E(y,\omega)-\tilde{\pi}_n\right)\left(\left(\tilde{R}_n\right)^6\tilde{R}_{t+1,R_2}f_E(z,\omega)-\tilde{\pi}_n\right)\;dydz\right).\end{multline}  Since there exists $C=C(R_1)>0$ such that, for all $n\geq 0$, $$C\tilde{D}_n\geq 12\tilde{D}_n+2R_1+R,$$ and since, for each $x\in\mathbb{R}^d$ and $\omega\in\Omega$, $$-1\leq \left(\tilde{R}_n\right)^6\tilde{R}_{t+1,R_2}f_E(y,\omega)-\tilde{\pi}_n\leq 1,$$  we have, in view of (\ref{u_var_000000}), for $C=C(R_1)>0$ independent of $n$, $k$ and $t$, $$\mathbb{E}\left(\left(\left(\overline{R}_n\right)^{k\ell_n^2-6}\left(\tilde{R}_n\right)^6\tilde{R}_{t+1,R_2}f_E(0,\omega)-\tilde{\pi}_n\right)^2\right) \leq \int_{\mathbb{R}^d}\int_{C\tilde{D}_n}p_{n,k}(y)p_{n,k}(z)\;dydz.$$  Therefore, using (\ref{u_var_2}), for $C=C(R_1)>0$ independent of $n$, $k$ and $t$, \begin{equation}\label{u_var_5}\mathbb{E}\left(\left(\left(\overline{R}_n\right)^{k\ell_n^2-6}\left(\tilde{R}_n\right)^6f_E(0,\omega)-\tilde{\pi}_n\right)^2\right)\leq Ck^{-d/2}L_{n+1}^{-d}\tilde{D}_n^d\leq Ck^{-d/2}\tilde{\kappa}_{n}^d\ell_n^{-d}, \end{equation} and, together with (\ref{u_var_5}), Chebyshev's inequality implies that, for $C=C(R_1)>0$ independent of $n$, $k$ and $t$, \begin{equation}\label{u_var_6} \mathbb{P}\left(\left|\left(\overline{R}_n\right)^{k\ell_n^2-6}\left(\tilde{R}_n\right)^6\tilde{R}_{t+1,R_2}f_E(0,\omega)-\tilde{\pi}_n\right|\geq 1/\tilde{\kappa}_n\right)\leq Ck^{-d/2}\tilde{\kappa}_{n}^{d+2}\ell_n^{-d}.\end{equation}  We will now extend a version of this estimate to the whole of $B_{4\sqrt{k}\tilde{D}_{n+1}}$.

Fix $0<\gamma<1$ satisfying, in view of (\ref{dimension}) and (\ref{Holderexponent}), \begin{equation}\label{u_var_7} \frac{1}{1+a}<\gamma<1,\;\;\textrm{which implies}\;\;\frac{2}{d}<\gamma<1.\end{equation}  Since $f_E$ satisfies (\ref{o_stationary}), the stationarity (\ref{stationary}) and (\ref{u_var_3}) imply that, for $C=C(R_1)>0$ independent of $n$, $k$, and $t$, \begin{multline*}\mathbb{P}\left(\sup_{x\in \left(\sqrt{k}L_{n+1}\right)^\gamma\mathbb{Z}^d\cap B_{4\sqrt{k}\tilde{D}_{n+1}}}\left|\left(\overline{R}_n\right)^{k\ell_n^2-6}\left(\tilde{R}_n\right)^6\tilde{R}_{t+1,R_2}f_E(x,\omega)-\tilde{\pi}_n\right|\geq 1/\tilde{\kappa}_n\right)  \\ \leq C\left(\frac{4\sqrt{k}\tilde{D}_{n+1}}{\left(\sqrt{k}L_{n+1}\right)^\gamma}\right)^dk^{-d/2}\tilde{\kappa}_{n}^{d+2}\ell_n^{-d}\leq C\tilde{\kappa}_{n+1}^d\tilde{\kappa}_n^{d+2}k^{-\gamma d/2}L_n^{(1-(1+a)\gamma)d},\end{multline*} where we observe that (\ref{u_var_7}) implies that $-\gamma d/2<-1$ and $(1-(1+a)\gamma)d<0$.  Therefore, in view of (\ref{L}) and (\ref{kappa}), there exists $\zeta>0$ and $C=C(R_1)>0$ independent of $n$, $k$, and $t$ such that  \begin{equation}\label{u_var_8} \mathbb{P}\left(\sup_{x\in \left(\sqrt{k}L_{n+1}\right)^\gamma\mathbb{Z}^d\cap B_{4\sqrt{k}\tilde{D}_{n+1}}}\left|\left(\overline{R}_n\right)^{k\ell_n^2-6}\left(\tilde{R}_n\right)^6\tilde{R}_{t+1,R_2}f_E(x,\omega)-\tilde{\pi}_n\right|\geq 1/\tilde{\kappa}_n\right)\leq Ck^{-(1+\zeta)}L_n^{-\zeta}.\end{equation}

Using Theorem \ref{effectivediffusivity}, Proposition \ref{o_sol} and (\ref{u_var_1}), for each $\omega\in\Omega$ and $x\in\mathbb{R}^d$, for $C>0$ independent of $n$, $k$, $t$ and $R_1$, \begin{equation}\label{u_var_9} \norm{D_x\left(\overline{R}_n\right)^{k\ell_n^2-6}\left(\tilde{R}_n\right)^6\tilde{R}_{t+1,R_2}f_E(x,\omega)}_{L^\infty(\mathbb{R}^d)} \leq \int_{\mathbb{R}^d}D_yp_{n,k}(y)\;dy\leq Ck^{-1/2}L_{n+1}^{-1}.\end{equation}  Since, for each $x\in B_{4\sqrt{k}\tilde{D}_{n+1}}$ there exists $y\in\left(\sqrt{k}L_{n+1}\right)^\gamma\mathbb{Z}^d\cap B_{4\sqrt{k}\tilde{D}_{n+1}}$ satisfying $\abs{x-y}\leq C\left(\sqrt{k}L_{n+1}\right)^\gamma$ for $C>0$ independent of $n$, we conclude that, in view of (\ref{u_var_9}), \begin{multline}\label{u_var_10} \sup_{x\in B_{4\sqrt{k}\tilde{D}_{n+1}}}\left|\left(\overline{R}_n\right)^{k\ell_n^2-6}\left(\tilde{R}_n\right)^6\tilde{R}_{t+1,R_2}f_E(x,\omega)-\tilde{\pi}_n\right| \\ \leq \sup_{x\in \left(\sqrt{k}L_{n+1}\right)^\gamma\mathbb{Z}^d\cap B_{4\sqrt{k}\tilde{D}_{n+1}}}\left|\left(\overline{R}_n\right)^{k\ell_n^2-6}\left(\tilde{R}_n\right)^6\tilde{R}_{t+1,R_2}f_E(x,\omega)-\tilde{\pi}_n\right|+C\left(\sqrt{k}L_{n+1}\right)^{\gamma-1}.\end{multline}  Because (\ref{L}) and (\ref{kappa}) imply that there exists $C>0$ independent of $n$ such that, for all $n\geq0$ and $k\geq 1$, we have \begin{equation}\label{u_var_11}\left(\sqrt{k}L_{n+1}\right)^{\gamma-1}\leq C/\tilde{\kappa}_n,\end{equation} for $C=C(R_1)>0$ independent of $n$, $k$, and $t$, using (\ref{u_var_8}), (\ref{u_var_10}) and (\ref{u_var_11}), \begin{equation}\label{u_var_12} \mathbb{P}\left(\sup_{x\in B_{4\sqrt{k}\tilde{D}_{n+1}}}\abs{\left(\overline{R}_n\right)^{k\ell_n^2-6}\left(\tilde{R}_n\right)^6\tilde{R}_{t+1,R_2}f_E(x,\omega)-\tilde{\pi}_n}>\frac{C}{\kappa_n}\right)\leq Ck^{-(1+\zeta)}L_n^{-\zeta}.\end{equation}  Finally, since (\ref{kappa}) and (\ref{D}) imply that there exists $C=C(t)>0$ such that, for $C_1>0$ as in (\ref{u_var_00}), for all $n\geq 0$,  $$e^{-\frac{(6\tilde{D}_n-C_1(1+t))_+^2}{C_1(1+t)}}\leq C/\tilde{\kappa}_n,$$ we conclude in view of (\ref{u_var_00}), (\ref{u_var_0000}) and (\ref{u_var_12}) that there exists $C=C(R_1,t)>0$ independent of $n$ and $k$ such that  $$\mathbb{P}\left(\sup_{x\in B_{4\sqrt{k}\tilde{D}_{n+1}}}\abs{\left(\overline{R}_n\right)^{k\ell_n^2-6}\left(\tilde{R}_n\right)^6R_{t+1}f_E(x,\omega)-\pi_n(R_tf_E)}>\frac{C}{\kappa_n}\right)\leq Ck^{-(1+\zeta)}L_n^{-\zeta},$$ which, since $E$, $R_1$, $n$, $k$ and $t$ were arbitrary, completes the argument.  \end{proof}

The following proposition is essentially a restatement of Proposition \ref{o_main} best suited to our current circumstances, where we recall the definition of the subsets $\left\{A_n\right\}_{n=0}^\infty$ in (\ref{o_subset}).

\begin{prop}\label{u_main}  Assume (\ref{steady}) and (\ref{constants}).  For every $E\in\mathcal{F}$, $n\geq 0$, $1\leq k<\ell_{n+1}^2$, $t\geq 0$ and $\omega\in A_n$, for $C>0$ independent of $E$, $n$, $k$, $t$ and $\omega$, $$\sup_{x\in B_{4\sqrt{k}\tilde{D}_{n+1}}}\abs{\left(R_{n+1}\right)^kR_{1+t}f_E(x,\omega)-\left(\overline{R}_n\right)^{k\ell_n^2-6}\left(\tilde{R}_n\right)^6R_{1+t}f_E(x,\omega)}\leq CL_n^{\beta-7(\delta-5a)}.$$\end{prop}

\begin{proof}  Fix $E\in\mathcal{F}$, $n\geq 0$, $1\leq k<\ell_{n+1}^2$ and $\omega\in A_n$.  In view of Proposition \ref{o_weak}, there exists $C>0$ independent of $E$, $n$, $k$, $t$ and $\omega$ such that $$\norm{R_{1+t}f_E}_{C^{0,\beta}(\mathbb{R}^d)}\leq C \norm{f_E}_{L^\infty(\mathbb{R}^d)}\leq C.$$  Therefore, since $\omega\in A_n$, Proposition \ref{o_main} implies that, for $C>0$ independent of $E$, $n$, $k$ and $\omega$, \begin{multline*}\sup_{x\in B_{4\sqrt{k}\tilde{D}_{n+1}}}\abs{\left(R_{n+1}\right)^kR_{1+t}f_E(x,\omega)-\left(\overline{R}_n\right)^{k\ell_n^2-6}\left(\tilde{R}_n\right)^6R_{1+t}f_E(x,\omega)} \\ \leq CL_n^{\beta-7(\delta-5a)}\norm{R_{1+t}f_E}_{C^{0,\beta}(\mathbb{R}^d)}\leq CL_n^{\beta-7(\delta-5a)},\end{multline*} which, since $E$, $n$, $k$, $t$ and $\omega$ were arbitrary, completes the argument.  \end{proof}

Observe that the convergence obtain in Proposition \ref{u_main} occurs along the discrete sequence of time steps $kL_n^2$ on $4\sqrt{k}\tilde{D}_n$.  We now upgrade this convergence along the full limit, as $t\rightarrow\infty$, using Control \ref{localization}.  The cost is that the convergence now occurs on a marginally smaller portion of space.

\begin{prop}\label{u_upmain}  Assume (\ref{steady}) and (\ref{constants}).  Suppose, for some $R_1>0$, $E\in\mathcal{F}$ satisfies $E\in \sigma_{B_{R_1}}.$  For each $n\geq 0$, $1\leq k<\ell_{n+1}^2$ and $t\geq 0$, for $C=C(R_1,t)>0$, $$\mathbb{P}\left(\sup_{kL_{n+1}^2\leq s\leq (k+1)L_{n+1}^2}\sup_{x\in B_{\sqrt{k}\tilde{D}_{n+1}}}\left| R_sR_{1+t}(x,\omega)-\pi_n(R_tf_E)\right|>\frac{C}{\tilde{\kappa}_n}\right) \leq Ck^{-(1+\zeta)}L_n^{-\zeta}.$$\end{prop}

\begin{proof}  Fix $E\in\mathcal{F}$ and $R_1>0$ satisfying $E\in \sigma_{B_{R_1}}$, $n\geq 0$, $1\leq k<\ell_{n+1}^2$ and $t\geq 0$.  We observe that, in view of (\ref{Holderexponent}), (\ref{kappa}) and (\ref{delta}), for each $n\geq 0$, for $\zeta>0$ defined in Proposition \ref{u_var}, there exists $C>0$ independent of $n$ satisfying, $$L_n^{\beta-7(\delta-5a)}\leq C/\tilde{\kappa}_n,$$ and, for each $n\geq 0$ and $1\leq k<\ell_{n+1}^2$, $$L_n^{(2(1+a)^2-1)d-M_0}\leq k^{-(1+\zeta)}L_n^{-\zeta}.$$   Therefore, Proposition \ref{o_probability}, Proposition \ref{u_var} and Proposition \ref{u_main} imply that, for $C=C(R_1,t)>0$ independent of $E$, $n$ and $k$, \begin{equation}\label{u_upmain_1} \mathbb{P}\left(\sup_{x\in B_{4\sqrt{k}\tilde{D}_{n+1}}}\left|\left(R_{n+1}\right)^kR_{1+t}f_E(x,\omega)-\pi_n(R_tf)\right|>\frac{C}{\tilde{\kappa}_n}\right)\leq Ck^{-(1+\zeta)}L_n^{-\zeta}.\end{equation}

Recall the cutoff function $\chi_{2\sqrt{k}\tilde{D}_{n+1}}$ defined in (\ref{cutoff}).  For each $x\in B_{\sqrt{k}\tilde{D}_{n+1}}$ and $kL_{n+1}^2\leq s\leq (k+1)L_{n+1}^2$, we write\begin{multline*} \abs{R_sR_{1+t}f_E(x,\omega)-\pi_n(R_tf_E)} =\left|R_{s-kL_{n+1}^2}\left(\left(R_{n+1}\right)^kR_{1+t}f_E(x,\omega)-\pi_n(R_tf_E)\right)\right|\leq \\ \left|R_{s-kL_{n+1}^2}\chi_{2\sqrt{k}\tilde{D}_{n+1}}\left(\left(R_{n+1}\right)^kR_{1+t}f_E(x,\omega)-\pi_n(R_tf_E)\right)\right| \\ +\left|R_{s-kL_{n+1}^2}(1-\chi_{2\sqrt{k}\tilde{D}_{n+1}})\left(\left(R_{n+1}\right)^kR_{1+t}f_E(x,\omega)-\pi_n(R_tf_E)\right)\right|.\end{multline*}  Since, for each $x\in\mathbb{R}^d$ and $\omega\in\Omega$, $$-1<\left(R_{n+1}\right)^kR_{1+t}f_E(x,\omega)-\pi_n(R_tf_E)<1,$$ we have, for each $\omega\in A_n$ and $x\in B_{\sqrt{k}\tilde{D}_{n+1}}$, using Control \ref{localization}, since $0\leq s-kL_{n+1}^2\leq L_{n+1}^2$, \begin{equation}\label{u_upmain_2} \left|R_{s-kL_{n+1}^2}(1-\chi_{2\sqrt{k}\tilde{D}_{n+1}})\left(\left(R_{n+1}\right)^kR_{1+t}f_E(x,\omega)-\pi_n(R_tf_E)\right)\right|\leq e^{-\kappa_{n+1}}.\end{equation}  Furthermore, for each $x\in\mathbb{R}^d$ and $\omega\in\Omega$, \begin{multline}\label{u_upmain_3} \left|R_{s-kL_{n+1}^2}\chi_{2\sqrt{k}\tilde{D}_{n+1}}\left(\left(R_{n+1}\right)^kR_{1+t}f_E(x,\omega)-\pi_n(R_tf_E)\right)\right| \\ \leq \sup_{x\in B_{4\sqrt{k}\tilde{D}_{n+1}}}\left|\left(R_{n+1}\right)^kR_{1+t}f_E(x,\omega)-\pi_n(R_tf_E)\right|.\end{multline}

To conclude, observe that (\ref{Holderexponent}), (\ref{L}), (\ref{kappa}) and (\ref{delta}) imply that there exists $C>0$ satisfying, for each $n\geq 0$, $$e^{-\kappa_n}\leq C/\tilde{\kappa}_n.$$  Therefore, because $kL_{n+1}^2\leq s< (k+1)L_{n+1}^2$ was arbitrary, using (\ref{u_upmain_2}) and (\ref{u_upmain_3}), for $C=R(R_1,t)>0$ independent of $E$, $n$ and $k$,  \begin{multline*}\mathbb{P}\left(\sup_{kL_{n+1}^2\leq s\leq (k+1)L_{n+1}^2}\sup_{x\in B_{\sqrt{k}\tilde{D}_{n+1}}}\abs{R_sR_{1+t}f_E(x,\omega)-\pi_n(R_tf_E)}>\frac{C}{\tilde{\kappa}_n}\right) \\ \leq \mathbb{P}\left(\sup_{x\in B_{4\sqrt{k}\tilde{D}_{n+1}}}\abs{\left(R_{n+1}\right)^kR_{1+t}f_E(x,\omega)-\pi_n(R_tf_E)}>\frac{C}{\tilde{\kappa}_n}\right)+\mathbb{P}\left(\Omega\setminus A_n\right).\end{multline*}  Since, for each $n\geq 0$ and $1\leq k< \ell_{n+1}^2$, $$L_n^{(2(1+a)^2-1)d-M_0}\leq k^{-(1+\zeta)}L_n^{-\zeta},$$  in view of Proposition \ref{o_probability} and (\ref{u_upmain_1}), for $C=C(R_1,t)>0$ independent of $E$, $n$ and $k$, $$\mathbb{P}\left(\sup_{kL_{n+1}^2\leq s\leq (k+1)L_{n+1}^2}\sup_{x\in B_{\sqrt{k}\tilde{D}_{n+1}}}\abs{R_sR_{1+t}f_E(x,\omega)-\pi_n(R_tf_E)}>\frac{C}{\tilde{\kappa}_n}\right) \leq Ck^{-(1+\zeta)}L_n^{-\zeta},$$ which completes the proof.  \end{proof}

We are now prepared to present the proof of our main result.  Define, for each $n\geq 0$, $1\leq k<\ell_{n+1}^2$, $t\geq 0$ and $E\in\mathcal{F}$ satisfying, for some $R_1>0$, $E\in \sigma_{B_{R_1}}$, for $C=C(R_1,t)>0$ as in Proposition \ref{u_upmain}, \begin{equation}\label{u_event_1}B_{n,k,t}(E)=\left\{\;\omega\in\Omega\;|\;\sup_{kL_{n+1}^2\leq s\leq (k+1)L_{n+1}^2}\sup_{x\in B_{\sqrt{k}\tilde{D}_{n+1}}}\left| R_sR_{1+t}f_E(x,\omega)-\pi_n(R_tf_E)\right|\leq \frac{C}{\kappa_n}\;\right\},\end{equation} and, for each $n\geq 0$, \begin{equation}\label{u_event} B_{n,t}(E)=\bigcap_{k=1}^{\ell_{n+1}^2-1}B_{n,k,t}(E).\end{equation}  We have, using Proposition \ref{u_upmain}, for each $n\geq 0$, $t\geq 0$, $E\in\mathcal{F}$ satisfying, for some $R_1>0$, $E\in\sigma_{B_{R_1}},$ for $C=C(R_1,t)>0$ independent of $n$ and $E$, $$\mathbb{P}\left(\Omega\setminus B_{n,t}(E)\right)\leq \sum_{k=1}^{\ell_{n+1}^2-1}\mathbb{P}\left(\Omega\setminus B_{n,k,t}(E)\right)\leq \sum_{k=1}^{\ell_{n+1}^2-1}Ck^{-(1+\zeta)}L_n^{-\zeta}\leq CL_{n}^{-\zeta}.$$  And, in view of (\ref{L}), for each $t\geq 0$ and $E\in\mathcal{F}$ satisfying $E\in\sigma_{B_{R_1}}$, for some $R_1>0$, $$\sum_{n=0}^\infty\mathbb{P}\left(\Omega\setminus B_{n,t}(E)\right)<\infty.$$  We therefore define, for each $t\geq 0$ and $E\in\mathcal{F}$ satisfying $E\in\sigma_{B_{R_1}}$, for some $R_1>0$, the subset \begin{equation}\label{o_mainevent} \Omega_{t}(E)=\left\{\;\omega\in\Omega\;|\;\textrm{There exists}\;\overline{n}(\omega)\geq0\;\textrm{such that, for all}\;n\geq\overline{n},\;\omega\in B_{n,t}.\;\right\},\end{equation} where the Borel-Cantelli lemma implies that, for every $t\geq 0$ and $E\in\mathcal{F}$ satisfying $E\in\sigma_{B_{R_1}}$, for some $R_1>0$, $$\mathbb{P}(\Omega_t(E))=1.$$

We now present the invariance property of the measure $\pi$.  In the proof, we use that fact that $$\mathcal{F}=\sigma\left(\bigcup_{R>0}\sigma_{B_R}\right).$$  That is, the sigma algebra $\mathcal{F}$ is generated by subsets satisfying the hypothesis of Proposition \ref{u_localize}.

\begin{prop}\label{u_invariant}  Assume (\ref{steady}) and (\ref{constants}).  For every $E\in\mathcal{F}$ and $t\geq 0$, $$\pi(E)=\overline{\pi}(R_tf_E)=\int_{\Omega}R_tf_E(0,\omega)\;d\pi.$$\end{prop}

\begin{proof}  We define \begin{equation}\label{u_invariant_1}\tilde{\mathcal{F}}=\bigcup_{R>0}\sigma_{B_R},\end{equation} and observe that $\tilde{\mathcal{F}}$ is an algebra of subsets of $\Omega$.  That is, $\tilde{\mathcal{F}}$ is closed under relative complements and finite unions.  Furthermore, for every $E\in\tilde{\mathcal{F}}$, there exists $R_1=R_1(E)>0$ satisfying $E\in\sigma_{B_{R_1}}$.

Fix $t\geq 0$ and $E\in\tilde{\mathcal{F}}$.  For every $\omega\in\Omega_t(E)\cap\Omega_0(E)$, (\ref{u_event_1}) and (\ref{u_event}) imply that $$\pi(E)=\overline{\pi}(f_E)=\lim_{s\rightarrow\infty}R_sR_1f_E(0,\omega)=\lim_{t\rightarrow\infty}R_sR_{t+1}f_E(0,\omega)=\overline{\pi}(R_tf_E).$$  And, in view of Proposition \ref{o_integral}, since $R_tf_E$ satisfies (\ref{o_stationary}), $$\overline{\pi}(R_tf_E)=\int_{\Omega}R_tf_E(0,\omega)\;d\pi.$$ Therefore, for every $t\geq 0$ and $E\in\tilde{\mathcal{F}}$, $$\pi(E)=\int_{\Omega}R_tf_E(0,\omega)\;d\pi.$$

To conclude, the absolute continuity of $\pi$ with respect to $\mathbb{P}$ and the dominated convergence theorem imply that, using a repetition of the argument appearing in Proposition \ref{o_pmeasure}, for each $t\geq 0$, the rule $$E\rightarrow \int_{\Omega}R_tf_E(0,\omega)\;d\pi$$ defines a probability measure on $(\Omega, \mathcal{F})$.  Therefore, since $$\mathcal{F}=\sigma(\tilde{\mathcal{F}}),$$ the Caratheodory Extension Theorem implies that, for every $E\in\mathcal{F}$ and $t\geq 0$, $$\pi(E)=\int_{\Omega}R_tf_E(0,\omega)\;d\pi,$$ which completes the argument.  \end{proof}

We now present our main result, which states the $\pi$ is the unique invariant measure which is absolutely continuous with respect to $\mathbb{P}$.

\begin{thm}\label{u_thm}  Assume (\ref{steady}) and (\ref{constants}).  There exists a unique invariant measure which is absolutely continuous with respect to $\mathbb{P}$. \end{thm}

\begin{proof}  The measure $\pi$ constructed according to (\ref{o_p}) was shown in Proposition \ref{o_pmeasure} to be absolutely continuous with respect to $\mathbb{P}$ and, was shown to be invariant in Proposition \ref{u_invariant}.  It therefore suffices to prove the uniqueness.

Suppose the $\mu$ is a probability measure on $(\Omega, \mathcal{F})$ which is absolutely continuous with respect to $\mathbb{P}$ and satisfies, for each $t\geq 0$ and $E\in\mathcal{F}$, \begin{equation}\label{u_thm_1}\mu(E)=\int_{\Omega}R_tf_E(0,\omega)\;d\mu.\end{equation}  Fix $E\in\tilde{\mathcal{F}}$.  In view of (\ref{u_event_1}) and (\ref{u_event}), for every $\omega\in \Omega_0(E)$, as $t\rightarrow\infty$, $$R_tf_E(0,\omega)\rightarrow \pi(E).$$  Furthermore, since $\mu$ is absolutely continuous with respect to $\mathbb{P}$, \begin{equation}\label{u_thm_2}\mu(\Omega_0(E))=\mathbb{P}(\Omega_0(E))=1.\end{equation}  Therefore, the dominated convergence theorem, (\ref{u_thm_1}) and (\ref{u_thm_2}) imply that $$\mu(E)=\lim_{t\rightarrow\infty}\int_\Omega R_tf_E(0,\omega)\;d\mu=\int_\Omega \pi(E)\;d\mu=\pi(E).$$  Since $E\in\tilde{\mathcal{F}}$ was arbitrary, the Caratheodory Extension Theorem implies, using $\mathcal{F}=\sigma(\tilde{\mathcal{F}})$, for every $E\in\mathcal{F}$, $$\pi(E)=\mu(E),$$ which completes the argument.  \end{proof}

In the final proposition of this section, we prove that if the transformation group \begin{equation}\label{u_ergodic}\left\{\tau_x\right\}_{x\in\mathbb{R}^d}\;\;\textrm{is ergodic,}\end{equation} then the invariant measure $\pi$ is mutually absolutely continuous with respect to $\mathbb{P}$ and is an ergodic probability measure for the canonical Markov process on $\Omega$ defining (\ref{i_initial}).

\begin{prop}\label{u_mutual}  Assume (\ref{steady}), (\ref{constants}) and (\ref{u_ergodic}).  The invariant measure $\pi$ is mutually absolutely continuous with respect to $\mathbb{P}$ and defines an ergodic probability measure for the canonical Markov process on $\Omega$ defining (\ref{i_initial}).\end{prop}

\begin{proof}  It was shown in Proposition \ref{o_pmeasure} that $\pi$ is absolutely continuous with respect to $\mathbb{P}$.  It remains to show that $\mathbb{P}$ is absolutely continuous with respect to $\pi$ on $(\Omega, \mathcal{F})$.

We proceed by contradiction, if not the Lebesgue decomposition theorem implies that there exists a subset $E\subset\Omega$ satisfying $0<\mathbb{P}(\Omega\setminus E)<1$ and $\pi(\Omega\setminus E)=0$, and with $\mathbb{P}$ absolutely continuous with respect to $\pi$ on $E$.  Since $$1=\pi(E)=\int_{\Omega}R_1f_E(0,\omega)\;d\pi=\int_\Omega\int_{\mathbb{R}^d}p(0,1,y,\omega){\bf{1}}_E(\tau_y\omega)\;dy\;d\pi,$$ and since $\pi$ is absolutely continuous with respect to $\mathbb{P}$, this implies that, for almost every $\omega\in E$ with respect to $\mathbb{P}$, and for almost every $y\in\mathbb{R}^d$, we have $\tau_y\omega\in E.$  Fubini's theorem therefore implies that, for almost every $y\in\mathbb{R}^d$, as elements of the measure algebra of $\mathcal{F}$, $$\tau_y(E)=E.$$  And, since the map $y\rightarrow {\bf{1}}_E(\tau_y\omega)$ is continuous from $\mathbb{R}^d$ to $L^1(\Omega)$, we conclude that, for every $y\in\mathbb{R}^d$, as elements of the measure algebra of $(\Omega, \mathcal{F},\mathbb{P})$, $$\tau_y(E)=E.$$  Therefore, using (\ref{u_ergodic}), we have $\mathbb{P}(E)=0$ or $\mathbb{P}(E)=1$, a contradiction.  Since $E$ was arbitrary, this completes the proof of mutual absolute continuity.

We now prove the ergodicity.  Suppose that, for $E\in\mathcal{F}$ and $t\geq 0$, we have $$R_tf_E(0,\omega)=P_t(\omega,E)=1\;\;\textrm{for almost every}\;\;\omega\in E\;\;\textrm{with respect to}\;\;\pi.$$  Since $\mathbb{P}$ is absolutely continuous with respect to $\pi$, this implies that $$R_tf_E(0,\omega)=P_t(\omega,E)=1\;\;\textrm{for almost every}\;\;\omega\in E\;\;\textrm{with respect to}\;\;\mathbb{P},$$  which, by repeating the above argument, implies that $E\in\mathcal{F}$ is an invariant set under the transformation group $\left\{\tau_x\right\}_{x\in\mathbb{R}^d}$.  Therefore, because (\ref{u_ergodic}) implies that $\mathbb{P}(E)=0$ or $\mathbb{P}(E)=1$, we conclude that, since $\pi$ is absolutely continuous with respect to $\mathbb{P}$, either $\pi(E)=0$ or $\pi(E)=1$.  Since $E$ was arbitrary, this completes the argument.  \end{proof}

\section{A Proof of Homogenization for Locally Measurable Functions}

In this section, we characterize the limiting behavior, as $\epsilon\rightarrow 0$, on a subset of full probability, of solutions $u^\epsilon:\mathbb{R}^d\times[0,\infty)\times\Omega\rightarrow\mathbb{R}$ satisfying \begin{equation}\label{h_eq}\left\{\begin{array}{ll} u^\epsilon_t-\frac{1}{2}\tr(A(x/\epsilon,\omega)D^2u^\epsilon)+\frac{1}{\epsilon}b(x/\epsilon,\omega)\cdot Du^\epsilon=0 & \textrm{on}\;\;\mathbb{R}^d\times(0,\infty), \\ u^\epsilon(x,0,\omega)=f(x/\epsilon,\omega) & \textrm{on}\;\;\mathbb{R}^d\times\left\{0\right\},\end{array}\right.\end{equation} for initial data $f\in L^\infty(\mathbb{R}^d\times\Omega)$ satisfying (\ref{o_stationary}) and, for some $R>0$, $f(0,\omega)\in L^\infty(\Omega,\sigma_{B_R}).$  In what follows, for each $E\in \bigcup_{R>0}\sigma_{B_R}$, recall the subset of full probability $\Omega_0(E)$ defined in (\ref{o_mainevent}), and the related subsets $B_{n,0}(E)$ defined in (\ref{u_event}).

\begin{thm}\label{h_thm} Assume (\ref{steady}) and (\ref{constants}).  Suppose $f\in L^\infty(\mathbb{R}^d\times\Omega)$ satisfies (\ref{o_stationary}) and, for some $R_1>0$, satisfies $f(0,\omega)\in L^\infty(\Omega,\sigma_{B_{R_1}})$.  For each $\epsilon>0$, write $u^\epsilon$ for the solution of (\ref{h_eq}) with initial data $f(x/\epsilon,\omega)$.  On a subset of full probability, as $\epsilon\rightarrow 0$, $$u^\epsilon\rightarrow\overline{\pi}(f)=\int_{\Omega}f(0,\omega)\;d\pi\;\;\textrm{locally uniformly on}\;\;\mathbb{R}^d\times(0,\infty).$$\end{thm}

\begin{proof}  Observe that, for each $n\geq 0$, $f,g\in L^\infty(\mathbb{R}^d\times\Omega)$ satisfying (\ref{o_stationary}) and $\alpha,\beta\in\mathbb{R}$, we have $$\pi_n(\alpha f+\beta g)=\alpha\pi_n(f)+\beta\pi_n(g)\;\;\textrm{and}\;\;\overline{\pi}(\alpha f+\beta g)=\alpha\overline{\pi}(f)+\beta\overline{\pi}(g).$$  Furthermore, for each $f,g\in L^\infty(\mathbb{R}^d\times\Omega)$, for each $n\geq 0$, $$\abs{\pi_n(f)-\pi_n(g)}=\abs{\pi_n(f-g)}\leq\norm{f-g}_{L^\infty(\mathbb{R}^d\times\Omega)},$$ and $$\abs{\overline{\pi}(f)-\overline{\pi}(g)}=\abs{\overline{\pi}(f-g)}\leq\norm{f-g}_{L^\infty(\mathbb{R}^d\times\Omega)}.$$  It therefore suffices, using the definition fo the Lebesgue integral, to prove the theorem for translates under $\left\{\tau_x\right\}_{x\in\mathbb{R}^d}$ of indicator functions corresponding to locally measurable sets in $\mathcal{F}$.

Fix $R_1>0$ and  $E\in\sigma_{B_{R_1}}$. We write $u^\epsilon$ for the solution of (\ref{h_eq}) with initial data $f_E(x/\epsilon,\omega)$.  Observe that, for each $\epsilon>0$, $u^\epsilon(x,t,\omega)=u(x/\epsilon,t/\epsilon^2,\omega)$ for $u:\mathbb{R}^d\times[0,\infty)\times\Omega\rightarrow\mathbb{R}$ satisfying $$\left\{\begin{array}{ll} u_t-\frac{1}{2}\tr(A(y,\omega)D^2u)+b(y,\omega)\cdot Du=0 & \textrm{on}\;\;\mathbb{R}^d\times(0,\infty), \\ u(x,t,\omega)=f_E(x,\omega) & \textrm{on}\;\;\mathbb{R}^d\times\left\{0\right\}.\end{array}\right.$$  It therefore remains to prove that, on a subset of full probability, for each $R>0$, \begin{equation}\label{h_thm_1}\lim_{\epsilon\rightarrow 0}\sup_{(x,t)\in B_{R/\epsilon}\times[R^2/\epsilon^2,\infty)}\abs{u(x,t,\omega)-\overline{\pi}(f_E)}=0.\end{equation}  We now prove that (\ref{h_thm_1}) is satisfied for every $\omega\in\Omega_0(E)$, for every $R>0$.

Fix $R>0$ and $\omega\in\Omega_0(E)$.  Since $\omega\in\Omega_0(E)$, choose $\overline{\epsilon}_1>0$ such that, whenever $0<\epsilon<\overline{\epsilon}$, $n\geq 0$ and $1\leq k<\ell_n^2$ satisfy $$kL_n^2\leq R^2/\epsilon^2\leq (k+1)L_n^2,$$ we have $\omega\in B_{n,0}(E)$.  And, in view of (\ref{kappa}) and (\ref{D}), choose $\overline{\epsilon}_2>0$ such that, whenever $0<\epsilon<\overline{\epsilon}_2$ satisfies, for $n\geq 0$ and $1\leq k<\ell_n^2$, $$kL_n^2\leq R^2/\epsilon^2<(k+1)L_n,$$ we have $$(k+1)R^2L_n^2\leq k\tilde{D}_n^2.$$

Define $\overline{\epsilon}=\min\left\{\overline{\epsilon}_1,\overline{\epsilon}_2\right\}$ and observe that, whenever $0<\epsilon<\overline{\epsilon}$ satisfies, for $n\geq 0$ and $1\leq k<\ell_n^2$, $$kL_n^2\leq R^2/\epsilon^2\leq (k+1)L_n^2,$$ we have $\omega\in B_{n,0}(E)$ and $B_{R/\epsilon}\subset B_{\sqrt{k}\tilde{D}_n}$.  Therefore, Proposition \ref{u_upmain} implies that, whenever $0<\epsilon<\overline{\epsilon}$ satisfies, for $\overline{n}\geq 0$ and $1\leq \overline{k}<\ell_{\overline{n}}^2$, $$\overline{k}L_{\overline{n}}^2\leq R^2/\epsilon^2<(\overline{k}+1)L_{\overline{n}}^2,$$  we have, for $C>0$ independent of $\overline{n}$, $$\sup_{(x,t)\in B_{R/\epsilon}\times[(R^2+\epsilon^2)/\epsilon^2,\infty)}\abs{u(x,t,\omega)-\overline{\pi}(f_E)}\leq C/\tilde{\kappa}_{\overline{n}}+\sup_{n\geq\overline{n}}\abs{\pi_n(f)-\overline{\pi}(f)}.$$  This, in view of Proposition \ref{o_identify}, completes the argument, since $\overline{n}\rightarrow\infty$ as $\epsilon\rightarrow 0$, and since $R>0$, $R_1>0$ and $E\in\sigma_{B_{R_1}}$ were arbitrary.  \end{proof}

Finally, we remark that, as an immediate consequence of Theorem \ref{h_thm}, our methods also imply the homogenization of equations involving an oscillating righthand side, \begin{equation}\label{h_1eq} \left\{\begin{array}{ll} u^\epsilon_t-\frac{1}{2}\tr(A(x/\epsilon,\omega)D^2u^\epsilon)+\frac{1}{\epsilon}b(x/\epsilon)\cdot Du^\epsilon=f(x/\epsilon,\omega) & \textrm{on}\;\;\mathbb{R}^d\times(0,\infty), \\ u^\epsilon=0 & \textrm{on}\;\;\mathbb{R}^d\times\left\{0\right\},\end{array}\right.\end{equation} and, analogous time-independent problems, like\begin{equation}\label{h_2eq} u^\epsilon-\frac{1}{2}\tr(A(x/\epsilon,\omega)D^2u^\epsilon)+\frac{1}{\epsilon}b(x/\epsilon,\omega)\cdot Du^\epsilon=f(x/\epsilon,\omega)\;\;\textrm{on}\;\;\mathbb{R}^d,\end{equation} for initial data $f(x,\omega)$ satisfying (\ref{o_stationary}) with, for some $R>0$, $f(0,\omega)\in L^\infty(\Omega,\sigma_{B_R})$.  The following two theorems summarize the results.

\begin{thm}\label{h_1thm}  Assume (\ref{steady}) and (\ref{constants}).  Suppose $f\in L^\infty(\mathbb{R}^d\times\Omega)$ satisfies (\ref{o_stationary}) and, for some $R_1>0$, satisfies $f(0,\omega)\in L^\infty(\Omega,\sigma_{B_{R_1}})$.  For each $\epsilon>0$, write $u^\epsilon$ for the solution of (\ref{h_1eq}) with righthand side $f(x/\epsilon,\omega)$.  On a subset of full probability, as $\epsilon\rightarrow 0$, $$u^\epsilon(x,t,\omega)\rightarrow t\overline{\pi}(f)=t\int_{\Omega}f(0,\omega)\;d\pi\;\;\textrm{locally uniformly on}\;\;\mathbb{R}^d\times[0,\infty).$$\end{thm}

\begin{proof}  Fix $f\in L^\infty(\mathbb{R}^d\times\Omega)$ and $R_1>0$ satisfying (\ref{o_stationary}) and $f(0,\omega)\in L^\infty(\Omega,\sigma_{B_{R_1}}).$  For each $\epsilon>0$, let $u^\epsilon:\mathbb{R}^d\times[0,\infty)\times\Omega\rightarrow\mathbb{R}$ satisfy (\ref{h_1eq}) with righthand side $f(x/\epsilon,\omega)$ and let $\tilde{u}^\epsilon:\mathbb{R}^d\times[0,\infty)\times\Omega\rightarrow\mathbb{R}$ satisfy (\ref{h_eq}) with initial data $f(x/\epsilon,\omega).$  We observe that, for each $\epsilon>0$, $$u^\epsilon(x,t,\omega)=\int_{0}^t\tilde{u}^\epsilon(x,s,\omega)\;ds\;\;\textrm{on}\;\;\mathbb{R}^d\times[0,\infty)\times\Omega.$$  This, in view of Theorem \ref{h_thm}, completes the proof.\end{proof}

\begin{thm}\label{h_2thm}  Assume (\ref{steady}) and (\ref{constants}).  Suppose $f\in L^\infty(\mathbb{R}^d\times\Omega)$ satisfies (\ref{o_stationary}) and, for some $R_1>0$, satisfies $f(0,\omega)\in L^\infty(\Omega,\sigma_{B_{R_1}})$.  For each $\epsilon>0$, write $u^\epsilon$ for the solution of (\ref{h_2eq}) with righthand side $f(x/\epsilon,\omega)$.  On a subset of full probability, as $\epsilon\rightarrow 0$, $$u^\epsilon\rightarrow \overline{\pi}(f)=\int_{\Omega}f(0,\omega)\;d\pi\;\;\textrm{locally uniformly on}\;\;\mathbb{R}^d.$$\end{thm}

\begin{proof}  Fix $f\in L^\infty(\mathbb{R}^d\times\Omega)$ and $R_1>0$ satisfying (\ref{o_stationary}) and $f(0,\omega)\in L^\infty(\Omega,\sigma_{B_{R_1}}).$  For each $\epsilon>0$, let $u^\epsilon:\mathbb{R}^d\times\Omega\rightarrow\mathbb{R}$ satisfy (\ref{h_2eq}) with righthand side $f(x/\epsilon,\omega)$ and let $\tilde{u}^\epsilon:\mathbb{R}^d\times[0,\infty)\times\Omega\rightarrow\mathbb{R}$ satisfy (\ref{h_eq}) with initial data $f(x/\epsilon,\omega).$  We observe that, for each $\epsilon>0$, $$u^\epsilon(x,\omega)=\int_{0}^\infty e^{-s}\tilde{u}^\epsilon(x,s,\omega)\;ds\;\;\textrm{on}\;\;\mathbb{R}^d\times\Omega.$$  This, in view of Theorem \ref{h_thm}, completes the proof.\end{proof}

\bibliography{measure}
\bibliographystyle{plain}

\end{document}